\RequirePackage[l2tabu,orthodox]{nag}		
\documentclass[reqno]{amsart}		
\usepackage[margin=1.5in,bottom=1.25in]{geometry}		

\usepackage{amsmath}		
\usepackage{amssymb}		
\usepackage{amsfonts}		
\usepackage{amsthm}		
\usepackage[foot]{amsaddr}		

\usepackage{mathtools}		

\mathtoolsset{%
}

\usepackage[utf8]{inputenc}		
\usepackage[T1]{fontenc}		

\usepackage[proportional,tabular,lining,sf,mono=false]{libertine}

\usepackage{dsfont}		

\usepackage[
cal=cm,
]
{mathalfa}

\usepackage{acronym}		

\usepackage[labelfont={bf,small},labelsep=colon,font=small]{caption}	
\usepackage[dvipsnames,svgnames]{xcolor}		
\colorlet{MyRed}{Crimson!90!Black}
\colorlet{MyBlue}{MediumBlue!90!Black}
\colorlet{MyGreen}{DarkGreen!80!Black}

\usepackage{pifont}

\usepackage{subcaption}		
\usepackage{tikz}		
\usetikzlibrary{calc,patterns,positioning}

\usepackage{pgfplots}
\pgfplotsset{compat=1.18}
\usepackage{forest}

\usepackage{booktabs}		
\usepackage[inline,shortlabels]{enumitem}		
\setenumerate{itemsep=\smallskipamount,topsep=\medskipamount}
\setitemize{itemsep=\smallskipamount,topsep=\medskipamount}

\usepackage{multirow}

\usepackage{floatrow}
\newfloatcommand{capbtabbox}{table}[][\FBwidth]

\usepackage[kerning=true]{microtype}		

\usepackage{xspace}		

\usepackage[numbers,sort&compress]{natbib}		

\bibpunct[, ]{[}{]}{,}{n}{,}{,}

\usepackage{hyperref}
\hypersetup{
colorlinks=true,
linktocpage=true,
pdfstartview=FitH,
breaklinks=true,
pdfpagemode=UseNone,
pageanchor=true,
pdfpagemode=UseOutlines,
plainpages=false,
bookmarksnumbered,
bookmarksopen=false,
bookmarksopenlevel=1,
hypertexnames=true,
pdfhighlight=/O,
urlcolor=MyBlue,linkcolor=MyBlue,citecolor=MyBlue,	
pdftitle={},
pdfauthor={},
pdfsubject={},
pdfkeywords={},
pdfcreator={pdfLaTeX},
pdfproducer={LaTeX with hyperref}
}

\usepackage[linesnumbered,ruled,vlined]{algorithm2e}

\SetCommentSty{mycommfont} 

\usepackage[sort&compress,capitalize,nameinlink]{cleveref}		
\crefname{algorithm}{Algorithm}{Algorithms}
\crefname{equation}{Eq.}{Eqs.}

\usepackage{thmtools}		
\usepackage{thm-restate}		

\theoremstyle{plain}
\newtheorem{theorem}{Theorem}		
\newtheorem{lemma}{Lemma}		
\newtheorem{proposition}{Proposition}		

\newtheorem*{corollary*}{Corollary}		

\theoremstyle{definition}
\newtheorem{definition}{Definition}		
\newtheorem{assumption}{Assumption}		
\newtheorem{example}{Example}		

\newtheorem*{definition*}{Definition}		
\newtheorem*{assumption*}{Assumptions}		
\newtheorem*{example*}{Example}		

\theoremstyle{remark}
\newtheorem{remark}{Remark}		
\newtheorem*{remark*}{Remark}		

\newcommand{\debug}[1]{#1}		

\usepackage{color-edits}
\addauthor{gb}{blue}
\addauthor{jn}{Orange}
\addauthor{ja}{pink}

\usepackage{siunitx}
\newcommand{\scinum}[1]{\num[round-precision=2,round-mode=figures,
     scientific-notation=true]{#1}}

\usepackage{thmtools}
\usepackage{thm-restate} 

\newcommand{\newmacro}[2]{\newcommand{#1}{\debug{#2}}}		
\newcommand{\newop}[2]{\DeclareMathOperator{#1}{\debug{#2}}}		

\newmacro{\defeq}{\triangleq}
\newmacro{\eqdef}{\triangleq}

\newcommand{\eg}{e.g.,}		
\newcommand{\ie}{i.e.,}		

\newcommand{\textpar}[1]{\textup(#1\textup)}		

\newcommand{\alt}[1]{#1'}		

\newcommand{\bbR}{\mathbb{R}}		

\newcommand{\bbN}{\mathbb{N}}		

\ifdef{\C}{                          
    \renewcommand{\C}{\mathcal{C}}
}{
    \newcommand{\C}{\mathcal{C}}
}

\newcommand{\CS}{\mathcal{S}}
\DeclareMathOperator{\bigoh}{\mathcal O}		

\newmacro{\dd}{\:d}		
\newcommand{\eps}{\varepsilon}		
\renewcommand{\epsilon}{\varepsilon}		

\newop{\dom}{dom}		

\newmacro{\set}{\mathcal{S}}		

\newmacro{\points}{\mathcal{K}}		
\newmacro{\point}{x}		
\newmacro{\pointalt}{\alt\point}		

\newmacro{\dpoints}{\mathcal{Y}}		
\newmacro{\dpoint}{y}		
\newmacro{\dpointalt}{\alt\dpoint}		

\newmacro{\base}{p}		
\newmacro{\basealt}{q}		

\DeclareMathOperator{\diam}{diam}		
\DeclareMathOperator{\dist}{dist}		

\newmacro{\open}{\mathcal{U}}           
\newmacro{\cpt}{\mathcal{K}}            
\newmacro{\U}{U}                        

\newmacro{\start}{1}		
\newmacro{\running}{1,2,\dotsc}		
\newmacro{\run}{t}		
\newmacro{\runalt}{s}		
\newmacro{\nRuns}{T}		
\newmacro{\runtime}{S}		

\newmacro{\runs}{\mathcal{\nRuns}}		

\newmacro{\vecspace}{\mathcal{X}}		
\newmacro{\vdim}{n}		
\newmacro{\vvec}{v}		
\newmacro{\bvec}{e}		
\newmacro{\unitvec}{z}		

\newmacro{\subspace}{\mathcal{W}}		
\newmacro{\subdim}{m}		

\newmacro{\tanspace}{\mathcal{Z}}		
\newmacro{\tvec}{z}		

\newcommand{\norm}[1]{\left\lVert#1\right\rVert}

\newmacro{\dspace}{\vecspace^{\ast}}		
\newmacro{\dvec}{w}		
\newmacro{\dbvec}{\eps}		

\newmacro{\ones}{\mathbf{1}}                     
\newmacro{\mat}{A}                               
\newmacro{\eye}{I}                               

\newop{\Tr}{Tr}

\newmacro{\cvx}{\mathcal{C}}      
\newmacro{\subd}{\partial}        

\newmacro{\strong}{\alpha}        
\newmacro{\smooth}{\beta}         

\newmacro{\obj}{f}		
\newmacro{\objalt}{g}		
\newmacro{\sobj}{F}		
\newmacro{\gvec}{g}		
\newmacro{\gbound}{G}		

\newmacro{\param}{\theta}		
\newmacro{\params}{\Theta}		

\newmacro{\oper}{A}		
\newmacro{\vecfield}{V}		
\newmacro{\vbound}{L}		

\newmacro{\M}{\mathcal{M}}		
\newmacro{\Mcol}{\mathcal{W}}		

\newmacro{\vx}{x}

\newmacro{\vy}{y}
\newmacro{\vxalt}{x'}
\newmacro{\vyalt}{y'}

\newmacro{\vxman}{\vx^{\M}}

\newmacro{\vn}{v_n}
\newmacro{\vnalt}{\hat{v}_n}

\newmacro{\vsg}{v} 

\newmacro{\tangentBundle}{T\mathcal{B}}

\newmacro{\maneq}{h}

\newmacro{\smoothcurve}{  c}
\newmacro{\geocurve}{  \gamma}

\newmacro{\grad}{  \operatorname{grad}}
\newmacro{\Hess}{  \operatorname{Hess}}

\newmacro{\R}{  \operatorname{R}}           
\newmacro{\D}{  \operatorname{D}}           
\newmacro{\Jac}{  \operatorname{Jac}}           

\newcommand{\projM}[1][\M]{\proj_{#1}}

\newmacro{\fun}{  F}   
\newmacro{\funalt}{ \tilde{F}}   
\newmacro{\funman}{  f} 
\newmacro{\funs}{  f}   
\newmacro{\funns}{  g}  
\newmacro{\funcomp}{  F}  
\newmacro{\mapping}{ c}

\newop{\proj}{\pi}
\newop{\prox}{prox}
\newop{\ri}{ri}
\newop{\Conv}{Conv}
\newop{\Aff}{Aff}
\newop{\Par}{Par}
\newop{\Diag}{Diag}
\newop{\diag}{diag}
\newop{\trace}{trace}
\newop{\bnd}{bnd}
\newop{\rbd}{rbd}

\newmacro{\ball}{\mathcal{B}}

\newmacro{\ndim}{n}
\newmacro{\inputSpace}{\bbR^{\ndim}}

\newmacro{\vxinter}{\vy}
\newmacro{\Minter}{\M^{\funns}}
\newmacro{\Mopt}{\opt[\M]}
\newmacro{\Mr}{\M^{\lambda_{\max}}_{\mult}}
\newmacro{\MI}{\M^{\max}_I}
\newmacro{\Mell}{\M_{I}^{\ell_{1}}}

\newmacro{\lammax}{\lambda_{\max{}}}
\newmacro{\mult}{r}

\newcommand{\opt}[1][\vx]{\debug{{#1}^\star}}		

\newmacro{\cm}{ e}

\newmacro{\stepLowBnd}{\varphi}
\newmacro{\stepUpBnd}{\Gamma}
\usepackage{accents}                                  
\newmacro{\stepLow}{\underaccent{\bar}{\step}}
\newmacro{\stepUp}{\bar{\step}}

\newmacro{\Mcodim}{p}
\newmacro{\maneqSpace}{\bbR^{\Mcodim}}

\newmacro{\dSQP}{d^{\mathrm{SQP}}}
\newmacro{\currdSQP}{d_{\ite}^{\mathrm{SQP}}(\curr[\vx])}

\newmacro{\dSQPt}{d^{\mathrm{SQP}}_{t}}
\newmacro{\dSQPn}{d^{\mathrm{SQP}}_{n}}
\newmacro{\dSQPr}{d^{\mathrm{SQP}}_{r}}

\newmacro{\gred}{\grad_{\M} \funcomp}
\newmacro{\redg}{\grad_{\M} \funcomp}
\newmacro{\redH}{\Hess_{\M} \funcomp}

\newmacro{\hessLag}{\nabla^2_{xx} L}

\newmacro{\dcorr}{d^{\mathrm{corr}}}
\newmacro{\Lagmult}{\lambda}

\newmacro{\lsstep}{\alpha}
\newmacro{\lsArmijoParam}{m}

\newmacro{\Msqp}{M}

\newmacro{\stepinit}{\step_{\text{init}}}

\newmacro{\minmax}{\Phi}		
\newmacro{\sadobj}{\minmax}		

\newmacro{\minvar}{\point_{1}}		
\newmacro{\minvaralt}{\alt\minvar}		
\newmacro{\minvars}{\points_{1}}		

\newmacro{\maxvar}{\point_{2}}		
\newmacro{\maxvars}{\points_{2}}		
\newmacro{\maxvaralt}{\alt\maxvar}		

\newmacro{\play}{i}		
\newmacro{\playalt}{j}		
\newmacro{\nPlayers}{N}		
\newmacro{\players}{\mathcal{\nPlayers}}		

\newmacro{\pure}{a}		
\newmacro{\purealt}{a'}		
\newmacro{\nPures}{n}		
\newmacro{\pures}{\mathcal{A}}		

\newmacro{\cost}{c}		
\newmacro{\loss}{\ell}		
\newmacro{\pay}{u}		
\newmacro{\payv}{v}		
\newmacro{\pot}{F}		

\newmacro{\game}{\mathcal{G}}		
\newmacro{\gamefull}{\game(\players,\points,\pay)}		

\newmacro{\fingame}{\Gamma}		
\newmacro{\fingamefull}{\Gamma(\players,\pures,\pay)}		

\newmacro{\mixgame}{\simplex(\fingame)}		

\newmacro{\corstrat}{\pi}		
\newmacro{\corprob}{\chi}		
\newmacro{\cormarg}{\point}		
\newmacro{\corprobs}{\points_{\mathrm{c}}}

\DeclareMathOperator{\prob}{\mathbb{P}}		
\DeclareMathOperator{\simplex}{\Delta}		

\newmacro{\sample}{\omega}		
\newmacro{\samples}{\Omega}		
\newmacro{\filter}{\mathcal{F}}		
\newmacro{\probspace}{(\samples,\filter,\prob)}		

\newmacro{\mean}{\mu}		
\newmacro{\sdev}{\sigma}		
\newmacro{\variance}{\sdev^{2}}		

\newmacro{\dkl}{D_{\mathrm{KL}}}		
\newmacro{\as}{\textpar{a.s.}\xspace}		

\newmacro{\hreg}{h}		
\newmacro{\breg}{D}		
\newmacro{\proxmap}{P}		
\newmacro{\mirror}{Q}		
\newmacro{\fench}{F}		
\newmacro{\hstr}{K}		
\newmacro{\depth}{H}		
\newmacro{\radius}{R}		
\newmacro{\zone}{\mathbb{D}}		
\newmacro{\subpoints}{\points^{\circ}}		

\newmacro{\state}{X}		
\newmacro{\dstate}{Y}		
\newmacro{\signal}{V}		

\newmacro{\step}{\gamma}		
\newmacro{\learn}{\eta}		

\newmacro{\ite}{k}
\newmacro{\initite}{0}
\newmacro{\afterinitite}{1}
\newmacro{\sumite}{n}
\newmacro{\Afterite}{K}

\newcommand{\curr}[1][\vx]{\debug{#1_{\ite}}}				

\newmacro{\graph}{\mathcal{G}}
\newmacro{\vertices}{\mathcal{V}}
\newmacro{\edges}{\mathcal{E}}

\newcommand{\Ranexp}{$\bbR_{\text{an,exp.}}$}       

\newmacro{\nPieces}{k}
\newmacro{\polyDeg}{N}
\newmacro{\smoothDeg}{m}

\newmacro{\Cm}{\C^\smoothDeg}
\newmacro{\cancube}{\appDom}

\newmacro{\nslice}{L}
\newmacro{\nCuts}{l}
\newmacro{\appDom}{A} 
\newmacro{\appDomRadius}{r_{\appDom}}

\newcommand{\polySpace}[1][\polyDeg]{\debug{\mathcal{P}}_{#1}}

\newcommand{\piecePolySpaceCuts}[1][\nCuts]{\debug{\mathcal{\widetilde{P}}}_{\polyDeg}^{#1}}

\DeclareMathOperator{\diffSet}{diff}
\newcommand{\diff}[1][\smoothDeg]{\debug{\diffSet}_{#1}}

\newmacro{\cube}{c}
\newmacro{\cubes}{\mathcal{C}_p}
\newmacro{\cubesU}{\mathcal{C}^{U}_p}
\newmacro{\cubesV}{\mathcal{C}^{V}_p}

\newmacro{\const}{C}
\newmacro{\consta}{C'}
\newmacro{\constalt}{C'}
\newmacro{\constb}{C^{\circ}}

\newmacro{\Cnm}{C_{\ndim,\smoothDeg,\cancube}}
\newmacro{\Cjack}{C^{2}_{\ndim,\smoothDeg,\cancube}}
\newmacro{\Cfeff}{C^{1}_{\ndim,\smoothDeg}}
\newmacro{\Cstrat}{C^{3}}
\newmacro{\Cdist}{C^{4}}
\newmacro{\Cerr}{C^{5}}
\newmacro{\Cfbanach}{\bar{C}_f}

\newmacro{\bNode}{\mathcal{T}_{B}}
\newmacro{\lNode}{\mathcal{T}_{L}}

\newmacro{\nSample}{n_{\text{samp}}}

\newmacro{\Mvol}{{v}}
\newmacro{\Mrad}{\mathrm{rad}}

\newmacro{\fCone}{f^\text{cone}}
\newmacro{\sCone}{s^\text{cone}}
\newmacro{\rCone}{r^\text{cone}}

\newmacro{\pmid}{p^\mathrm{mid}}
\newmacro{\rCube}{\diam}

\newmacro{\polyboundary}{p^{boundary}}
\newmacro{\piece}{s}

\newmacro{\epsmarg}{\Cstrat \nCuts^{-\frac{2}{\ndim-1}}}
\newmacro{\Mbig}{\M_0^\piece}

\renewcommand{\resizebox}[3]{#3}
\usepackage{ifthen}

\newboolean{IsArxivVersion}
\setboolean{IsArxivVersion}{true}

\begin{document}


\title
[Piecewise Polynomial Regression of Tame Functions via Integer Programming]
{Piecewise Polynomial Regression of Tame Functions \\ via Integer Programming}

\author
[G.~Bareilles]
{Gilles Bareilles$^{\dagger}$}
\address{$^{\dagger}$\,%
Faculty of Electrical Engineering, Czech Technical University of Prague, the Czech Republic.}

\author
[J.~Aspman]
{Johannes Aspman$^{\dagger}$}

\author
[J.~N\v{e}me\v{c}ek]
{Ji\v{r}\'{i} N\v{e}me\v{c}ek$^{\dagger}$}

\author
[J.~Mare\v{c}ek]
{Jakub Mare\v{c}ek$^{\dagger}$}



\begin{abstract}
Tame functions are a class of nonsmooth, nonconvex functions,  
which feature in a wide range of applications: functions encountered in the training of deep neural networks with all common activations, 
 value functions of mixed-integer programs, or wave functions of small molecules.
We consider approximating tame functions with piecewise polynomial functions.  
We bound the quality of approximation of a tame function by a piecewise polynomial function with a given number of segments on any full-dimensional cube.
We also present the first mixed-integer programming formulation of piecewise polynomial regression.
Together, these can be used to estimate tame functions.
We demonstrate promising computational results.
\end{abstract}

\maketitle
\allowdisplaybreaks		
\acresetall		



\section{Introduction}
\label{sec:introduction}


In a wide range of applications, one encounters nonsmooth and nonconvex functions that are \emph{tame}, short for \emph{definable in o-minimal structures}.
Such functions can be decomposed into a finite number of regions, where the function is smooth across each region, as illustrated in \cref{fig:NN_intro} (left pane).
Tame functions appear in a broad range of useful and difficult, \ie{} nonsmooth and nonconvex, applications.
Prominent examples are all common (nonsmooth and nonconvex) deep learning architectures \citep{davisStochasticSubgradientMethod2020,bolteMathematicalModelAutomatic2020}, and (nonsmooth) empirical risk minimization frameworks \citep{iutzelerNonsmoothnessMachineLearning2020}.
Tame functions also appear, \eg{} in mixed-integer optimization, with the value function and the solution to the so-called subadditive dual \citep{aspman2023taming}; in quantum information theory, with approximations of the matrix exponential for a k-local Hamiltonian \citep{bondar2022recovering,aravanis2022polynomial}; and in quantum chemistry, with functions describing the electronic structure of molecules.
The tame assumption was key in recent advances in learning theory, with notably the first convergence proofs of Stochastic Gradient Descent  \citep{davisStochasticSubgradientMethod2020,davis2021active,bianchiStochasticSubgradientDescent2021}, or theory of automatic differentiation \citep{bolteMathematicalModelAutomatic2020}.
The tameness property of a function is, among other things, stable under composition. We will discuss these aspects in more detail later.

\begin{figure}
  \begin{center}
      \includegraphics[width=0.6\textwidth]{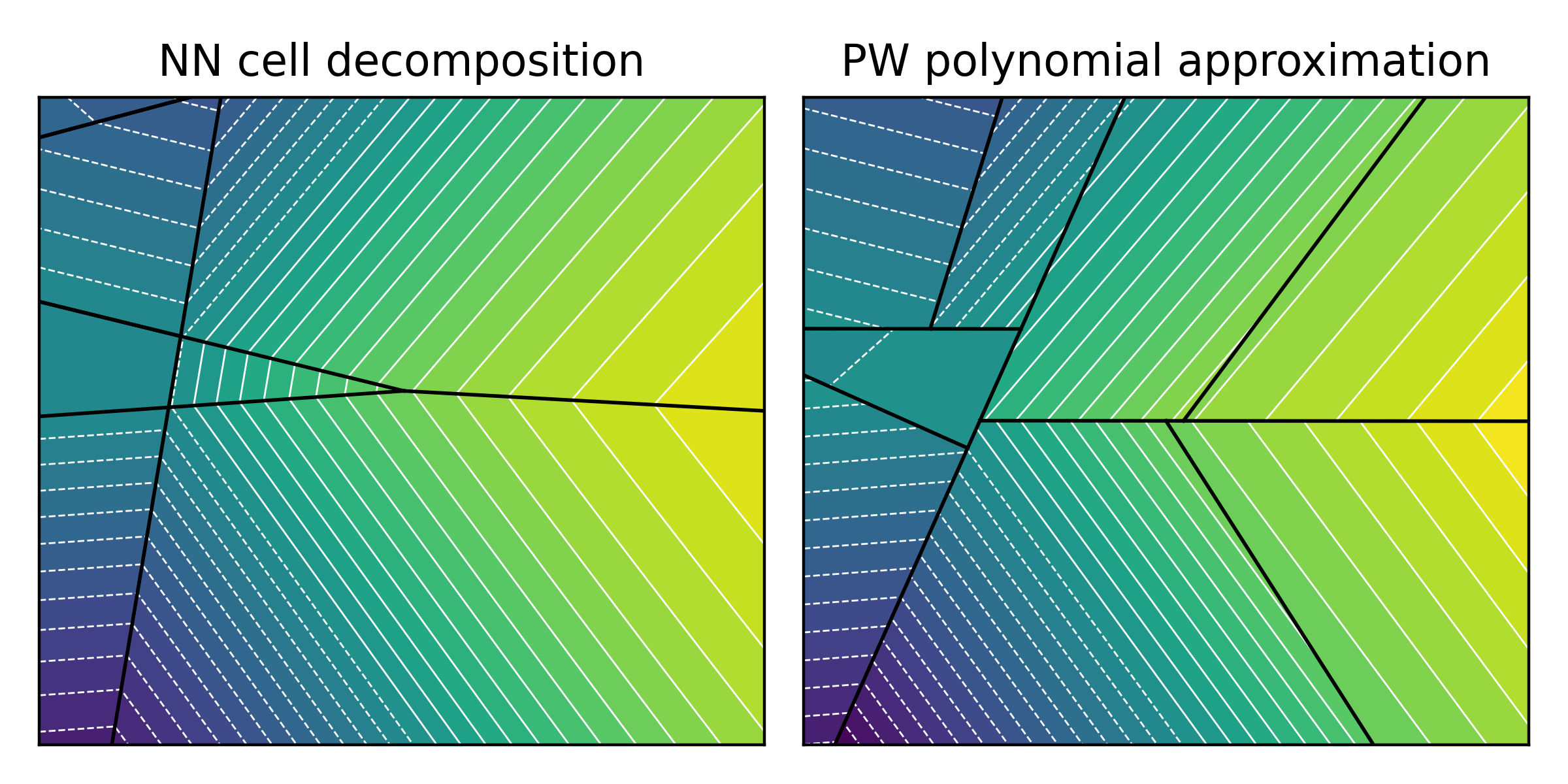}
    \caption{
      Left pane: a generic 2-dimensional function, with level lines (white) and nonsmooth points (black).
      Right pane: piecewise polynomial approximation of the network, obtained from the proposed integer program with a depth 3 regression tree and degree 2 polynomials.
    }
    \label{fig:NN_intro}
  \end{center}
  \vskip -0.2in
\end{figure}

This paper is concerned with building approximation of tame nonsmooth functions, a topic which has received attention recently  \citep{NEURIPS2021_dba4c1a1,chatterjeeAdaptiveEstimationMultivariate2021,donoho1997cart}.
Formally, given a function $f$, we seek a function $p$ in a set of simple functions $\mathcal{P}$ that minimizes the distance from $p$ to $f$ over a domain $A$:
\begin{equation*}
    \inf_{p\in\mathcal{P}} \left( \|f - p\|_{\infty, A} = \sup_{x\in A} |f(x) - p(x)|\right).
\end{equation*}

A major challenge is the nonsmoothness of the function to approximate.
Classical polynomial approximation theory shows that good (``fast'') polynomial approximation of a function is possible if the function has a high degree of regularity.
Specifically, if $f$ is a function $(\smoothDeg+1)$-times continuously differentiable, the best degree $\polyDeg$ polynomial incurs an $\bigoh(1/\polyDeg^{\smoothDeg})$ error \citep{plesniakMultivariateJacksonInequality2009}.
However, the situation changes dramatically when the function has low regularity, \eg{} is continuous but with discontinuous derivatives, as is the case for most settings in learning theory.
Indeed, approximating the absolute value --- the simplest nonsmooth function --- by degree $\polyDeg$ polynomials over the interval $[-1, 1]$ incurs \emph{exactly} the slow rate $1/\polyDeg$:
\begin{equation*}
  \inf_{p \in \mathcal{P}_{\polyDeg}} \|p - |\cdot|\|_{\infty, [-1, 1]} = \frac{\beta}{\polyDeg} + o\left(\frac{1}{\polyDeg}\right),
\end{equation*}
where $\beta\approx 0.28$ \citep{bernsteinMeilleureApproximationPar1914}.
In sharp contrast, allowing an approximation by \emph{piecewise polynomials} makes this problem simpler: the absolute value itself is a piecewise polynomial, consisting of two polynomial pieces of degree 1.
As a second example, consider a neural network comprised of sigmoid, $\tanh$, and ReLU activations.
Its landscape contains nonsmooth points, showed in black in \cref{fig:NN_intro} (left pane).
As the empirical risk minimization is tame, the nonsmooth points delineate regions of space where the function behaves smoothly.
We know of no convergent method for finding a piecewise polynomial approximation of tame functions.

\textbf{Related work.}
While simple, these examples illustrate the two fundamental challenges of estimating nonsmooth functions:
\begin{enumerate}[(i)]
\item estimating the \emph{cells} of the function, that is the full-dimensional sets on which the function is smooth; and, 
\item estimating the function on each cell.
\end{enumerate}
These challenges require input from several branches of mathematics. 


For \emph{(i)}, one can consult Model Theory, and, more specifically, o-minimal structures.  
Let us recall a central theorem there: the graph of any tame function splits into \emph{finitely} many full-dimensional sets, known as ``cells'', on which the function can have any desired degree of smoothness \citep{van1998tame}.
A natural approach is then to estimate this cell decomposition of the space.
\citet{rannouComplexityStratificationComputation1998} reduces this to a quantifier elimination problem and proposes a procedure that takes doubly-exponential time, for semialgebraic functions.
We are not aware of any implementation of this approach.
\citet{helmerConormalSpacesWhitney2023,helmerEffectiveWhitneyStratification2023} tested an implementation for real algebraic varieties and, possibly, semialgebraic sets.
We propose an algorithm that applies to general tame functions, thus covering the semialgebraic case but also networks using \eg{} sigmoid or $\tanh$ activations.

In Machine Learning, 
\citet{serraBoundingCountingLinear2018,pmlr-v97-hanin19a,liu2023relu} study the cells of networks built from ReLU activation and linear layers. They provide bounds on the number of cells, as well as a way to compute the cells of a given network by a mixed-integer linear program.
In contrast, we propose to approximate the tame function by a piecewise polynomial function such that each piece is defined by affine inequalities.
Thus, the whole domain is partitioned by a series of affine-hyperplane cuts, organized in a hierarchical tree structure, and a polynomial function is fit to each region.

To address \emph{(ii)}, Computational Statistics approximate smooth functions by polynomials on polytopes.
There, algorithms are mostly focused on continuous and typically one-dimensional functions \citep{goldbergMINLPFormulationsContinuous2021, warwickerComparisonTwoMixedInteger2021, pwlf, warwickerGeneratingOptimalRobust2023},
and often restricted to approximation of piecewise linear functions, \citep{vielmaMixedIntegerModelsNonseparable2010,KAZDA2021107310, huchetteNonconvexPiecewiseLinear2023}. 
Piecewise polynomial regression with polynomials of degree $\polyDeg \ge 2$ is either not addressed \cite{goldbergMINLPFormulationsContinuous2021,warwickerComparisonTwoMixedInteger2021,warwickerGeneratingOptimalRobust2023} or done through ``dimensionality lifting'' by appending values of all monomials of degrees $2$ to $\polyDeg$ as extra features to the individual samples \cite{pwlf}.
This approach of \citet{pwlf} requires us to know the nonsmooth points in advance.
We make no such assumption.


In Machine Learning, the use of trees befits the task of piecewise polynomial regression, with regression in dimensions higher than 1 utilizing hierarchical partitioning of the input space.
In \citet{NEURIPS2021_dba4c1a1}, the Dyadic CART algorithm of \citet{donoho1997cart} is used to recover a piecewise constant function defined on a lattice of points on the plane, with typical application in image processing and denoising.
This builds upon the work on Optimal Regression Trees (ORT) for classification by \citet{bertsimasOptimalClassificationTrees2017},
and is concerned with the optimal axis-aligned partitioning of the space, resulting in a piecewise constant function.
We stress that \citet{NEURIPS2021_dba4c1a1} only suggest a brute-force computation of the trees for piecewise constant functions. 

In Statistical Theory, \citet{chatterjeeAdaptiveEstimationMultivariate2021} reasons about sample complexity of piecewise polynomial regression via ORT, which also assumes that the data is defined over a lattice and that the splits are axis-aligned, but accommodates fitting polynomials of arbitrary degree and lattices of points in arbitrary dimensions $\ndim \ge 2$.
The lattice data structure assumption facilitates the possibility of using dynamic programming to solve the mixed-integer program, leading to polynomial-time complexity in the number of samples $\nSample$, but has never been demonstrated in practice. 
Two obvious shortcomings of the above approaches are that they require both that the data be defined over a regular lattice and that the splits of the tree are axis-aligned. We know of no proposal to compute optimal regression trees without requiring axis-aligned splits.
This significantly limits the type of piecewise polynomial functions they can reasonably fit, such as the neural network of \cref{fig:NN_intro}, the example function \eqref{eq:cone2dintro} illustrated in \Cref{fig:conelevelsjoint}, or even the simple $\|\cdot\|_{\infty}$ polyhedral norm.



\paragraph{Our contributions.}

In this work, we combine these results from model theory, approximation theory and optimization theory to present:
\begin{itemize}
  \item the first theoretical bound on the approximation error of generic \emph{nonsmooth and nonconvex} tame functions by piecewise polynomial functions, see \cref{th:main};
  \item a procedure to compute the best piecewise polynomial function, where each piece is a polyhedron defined by a number of affine inequalities, and the polynomial on each piece has a given degree.
\end{itemize}
The latter procedure involves sampling the function, and then solving the mixed-integer optimization problem formulated in \cref{eq:OrtFormulationHplane},
    which extends the OCT-H formulation of \citet{bertsimasOptimalClassificationTrees2017} to the regression task. Specifically, we propose a new MIP formulation for finding provably optimal regression trees with arbitrary hyperplane splits and polynomials of arbitrary degree in the partitions. In other words, the method finds optimal piecewise polynomial functions in any dimension and for any degree polynomials, with hierarchical partitioning of the space.
We find that such mixed-integer programs can be solved by current solvers to global optimality with one hundred samples within one hour to global optimality to yield precise piecewise polynomial approximations with affine-hyperplane splits.
   

\Cref{fig:NN_intro} shows the output of the proposed procedure on a tame neural network involving sigmoid, $\tanh$ and ReLU activations.
The left pane shows the level lines, and the nonsmooth points (in black)
that constitute the boundaries of the cells on which the function is smooth.
The right pane shows the approximation found using the formulation proposed in \Cref{sec:mip-form} after one hour of computation.

\textbf{Notation.}
$\C^{\smoothDeg}(A)$ is the set of $\smoothDeg$-times continuously differentiable functions from $A$ to $\bbR$. For a function $f:A\to\bbR$, we define $\diff(f)$ to be the subset of $A$ such that all of the $\smoothDeg$-th order partial derivatives of $f$ exist.
For a real-valued function $f:A\to\bbR$ of a set $A$, we let $\|f\|_{\infty,A} \eqdef \sup_{x\in A} |f(x)|$. 
For any positive integer $n\in\bbN$, $[n]$ denotes the set of integers from 1 to $n$.


\section{Background on tame geometry}
\label{sec:background}

In this section, we outline the main ideas and intuitions of tame geometry, and present the main result of interest: cell decomposition.
A more formal presentation is deferred to \Cref{appx:def}.

\subsection{Definability in o-minimal structures}
\label{sec:defin-o-minim}
An o-minimal structure is a collection of certain subsets of $\bbR^m$ that are stable under a large number of operations. One of the key properties that it should have is that when considering one-dimensional sets, they are only given by \emph{finite} unions of intervals and points. The structure is furthermore closed under Boolean operations, e.g. taking closures, unions or complements, as well as under projections to lower-dimensional sets, and elementary operations such as addition, multiplication and composition.

We use the words \emph{definable} and \emph{tame} interchangeably to refer to sets or functions that belong to a given o-minimal structure. Tame functions are in general nonsmooth and nonconvex, but the tameness properties still make it possible to have some control when studying the behaviour of these functions. The function class is also broad enough to entail most of the cases that would appear in applications in a vast number of fields. For example, in machine learning, most functions that would appear as activation functions in neural networks are tame.

Due to the balance between being wild enough to include a large number of non-trivial applications while still being tame enough such that it is possible to derive qualitative results on their behaviour, the interest in tame functions has seen a recent surge in numerous fields. 
In mathematical optimization, this framework allowed showing results that had proved elusive, notably on the convergence to critical points or escaping of saddle points for various (sub)gradient descent algorithms \citep{davisStochasticSubgradientMethod2020,joszGlobalConvergenceGradient2023,bianchiStochasticSubgradientDescent2021}.


\begin{example}[Analytic-exponential structure]
  Probably the most important o-minimal structure for applications is the one consisting of the analytic-exponential sets, which forms a structure typically denoted \Ranexp.
  An \emph{analytic-exponential set} in $\bbR^{\ndim}$ is defined as a finite union of sets of the form
  \begin{equation*}
          \{x\in \bbR^{\ndim} : \; f_1(x)=0,\dots,f_k(x)=0, g_1(x)>0,\dots, g_l(x)>0\},
  \end{equation*}
  where $k,l$ are finite integers, and $f_i$ and $g_i$ are functions obtained by combining the following:
  \begin{enumerate}[i.]
    \item coordinate functions $x \mapsto x_{i}$ and polynomial/semialgebraic functions,
    \item the restricted analytic functions: the functions $h:\bbR^{\ndim}\to\bbR$ such that $h|_{[-1,1]^{\ndim}}$ is analytic and $h$ is identically zero outside $[-1, 1]^{\ndim}$,
    \item the inverse function $x \mapsto 1/x$, with the convention that $1/0 = 0$,
    \item and the real exponential and logarithm function (the latter is extended to $\bbR$ by setting $\log(x) = 0$ for $x\le 0$). 
  \end{enumerate}

  In particular, any function built from the above rules is definable in the structure \Ranexp.
  Examples include almost all deep learning architectures, but also conic convex functions, wave functions of small molecules, or the following contrived functions  $(x, y) \mapsto x^{2}\exp(-\frac{y^{2}}{x^{4} + y^{2}})$, $(x, y) \mapsto x^{\sqrt{2}}\ln(\sin y)$ for $(x,y)\in(0, \infty)\times (0, \pi)$, or $(x, y) \mapsto y / \sin(x)$ for $x\in(0, \pi)$ \citep[Sec. 1]{kurdykaGradientsFunctionsDefinable1998}.
\end{example}

\subsection{Tame functions are piecewise \texorpdfstring{$\mathcal{C}^\smoothDeg$}{Cr}}
One of the central results in o-minimality theory is the cell decomposition theorem, and the related stratification theorems. These theorems give structural results on the graphs of tame functions by describing how the graph can be partitioned into smaller sets with some control on how the pieces fit together, as well as on the regularity of the function on each piece. In particular, it tells us that we can partition the graph into a finite number of pieces such that the function is $\mathcal{C}^{\smoothDeg}$, for any $\smoothDeg<\infty$, on each piece.  
We introduce the result in a simplified form that best suits our needs. More detailed statements are given in \Cref{appx:def}. 


\begin{proposition}[$\C^{\smoothDeg}$-cell decomposition]\label{prop:celldecomp}
  Fix an o-minimal expansion of $\bbR$.
  Consider a definable full-dimensional set $\appDom\subset\bbR^{\ndim}$ and a definable function $f:A \to \bbR$.
  Then, for any positive integer $\smoothDeg$, there exists a finite collection $\Mcol$ of sets $\M\subset\bbR^{\ndim}$, called cells, such that
  \begin{itemize}
    \item each cell $\M\in\Mcol$ is open, definable, full-dimensional,
    \item the sets of $\Mcol$ are pair-wise disjoint,
    \item $\appDom$ is the union of the closures of the elements of $\Mcol$,
    \item the restriction of $f$ to each cell $\M\subseteq\Mcol$ is $\mathcal{C}^{\smoothDeg}$.
  \end{itemize}
\end{proposition}
\begin{proof}
  This proposition is a direct corollary of the cell decomposition theorem and the Whitney stratification theorem of definable maps, recalled in \Cref{appx:def}, with $\Mcol$ corresponding to the set of all full-dimensional strata.
\end{proof}

\begin{figure}[t]
  \begin{center}
    \includegraphics[width=0.4\textwidth]{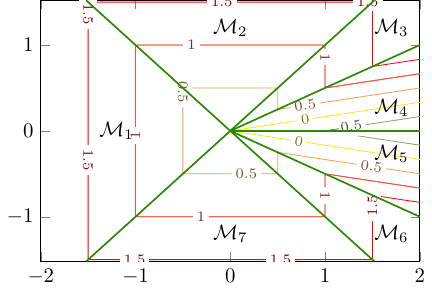}
    \caption{Illustration of the ``cone'' function \eqref{eq:cone2dintro}, with $\sCone=\rCone=0.5$, showing \emph{(i)} the level lines of the function, and \emph{(ii)} the decomposition of the domain into cells on which the function is smooth, as provided by \cref{prop:celldecomp}; see \cref{table:conecells} for details.}
    \label{fig:conelevelsjoint}
  \end{center}
  \vskip -0.2in
\end{figure}

\begin{example}\label{ex:conecelldecomp}
  Consider the following tame (piecewise linear) function of $\bbR^2$:
  \begin{equation}%
    \label{eq:cone2dintro}
      \fCone(x) =
      \begin{cases}
        -\sCone x_{1} + \frac{1+\sCone }{\rCone }x_{2} & \text{ if } x_{1} > 0 \text{ and } 0 < x_{2} < \rCone  x_{1} \\
        -\sCone x_{1} - \frac{1+\sCone }{\rCone }x_{2} & \text{ if } x_{1} > 0 \text{ and } -\rCone x_1 < x_{2} < 0 \\
        \|x\|_{\infty} & \text{ else}
      \end{cases}
  \end{equation}
  When $\rCone=\sCone=0.5$, as guaranteed by \cref{prop:celldecomp}, the domain of the function splits in $7$ cells $(\M_{i})_{i\in[7]}$, on which the function behaves smoothly.
  Here the pieces are polytopes, on which the function is linear.
  \Cref{fig:conelevelsjoint} shows the landscape of the function, and collection of cells $\Mcol = (\M_{i})_{i\in[7]}$.
  In addition, \cref{table:conecells} of \Cref{appx:def} summarizes the analytic expressions of the cells and of the (smooth) restriction of the function on each cell.
\end{example}




\section{Approximation of tame functions}
\label{sec:appr-defin-funct}


In this section, we state our main theoretical result: any tame function can be approximated to an arbitrary precision by a piecewise polynomial function.

\begin{definition}[Polynomial functions]
  We let $\polySpace$ denote the set of polynomials of degree at most $\polyDeg$.  
\end{definition}
\begin{definition}[Piecewise-polynomial functions]\label{def:piecepolycuts}
  We let $\piecePolySpaceCuts(\appDom)$ denote the set of functions that are piecewise polynomial on $\appDom$, such that
  \begin{enumerate}
    \item each piece $\piece$ is a polyhedron defined as the intersection of $\nCuts$ halfspaces, represented as one leaf of a complete binary tree of depth $\nCuts$ where each node collects an affine split of the space,
    \item the restriction of the function to each piece $\piece$ is a polynomial of degree at most $\polyDeg$.
\end{enumerate}
\end{definition}

We are now ready to state our main result, which combines the cell decomposition, \cref{prop:celldecomp}, with a more classical result from smooth function approximation theory.

\begin{restatable}[Main result]{theorem}{mainres}\label{th:main}
  Consider a function $f:\cancube\to\bbR$, and a constant $K>0$ such that:
  \begin{itemize}
    \item $f$ is definable in an o-minimal structure, and
    \item $f$ is $K$-Lipschitz on $\cancube$: for all $x$, $y\in\cancube$, $|f(x) - f(y)| \le K \|x-y\|$.
    \item $\appDom$ is a connected compact subset of $\inputSpace$, such that any two points $x$ and $y$ in $\appDom$ can be joined by a rectifiable arc in $\appDom$ with length no greater that $\sigma\|x-y\|$, where $\sigma$ is a positive constant.
  \end{itemize}
  Then $f$ is piecewise approximable by piecewise polynomial functions (see \cref{def:piecepolycuts}): for any integers $\nCuts\ge 1$, $\polyDeg\ge 1$, and $\smoothDeg > 1$,
  \begin{equation}\label{eq:them}
    \inf_{p \in \piecePolySpaceCuts(\cancube)} \| f -    p\|_{\infty, \appDom} \le C_1 \polyDeg^{-\smoothDeg} + 
C_2 \nCuts^{-\frac{2}{\ndim-1}}.
  \end{equation}
  where $C_1$ depends only on $\ndim$, $\smoothDeg$, $\cancube$, and $f$, and $C_2$ depends only on $\ndim$, $\smoothDeg$, and $f$.
\end{restatable}

\Cref{th:main} is our main approximation result so, before discussing its proof (which we carry out in detail in \Cref{sec:proof-th}), some remarks are in order.

Details of the dependence of the constants $C_1$ and $C_2$ on the problem parameters are given in \Cref{sec:proof-th}. 

The definability assumption ensures that the graph of the function does not oscillate arbitrarily. 
This allows us to use results from \cite{boissonnatTracingIsomanifoldsTime2021}, which do not hold for (not definable) $K$-Lipschitz functions. 


\paragraph{Proof outline of \Cref{th:main}.}
The bound is obtained by considering a piecewise linear approximation of the nonsmooth regions of the function constructed by intersecting $\nCuts$ halfplanes. The proof then splits into two parts, which correspond to the two terms in the bound in \eqref{eq:them}.

Firstly, each polyhedral piece will intersect one `large' (full-dimensional) cell, where the function is $\Cm$. This cell is provided by the cell decomposition theorem.
We can then make use of the Whitney extension theorem together with a Jackson-type theorem to get the first term of \eqref{eq:them}, which matches the fast approximation rates of smooth approximation theory. 

Secondly, a polyhedral piece may also intersect with an additional (or more) cells.
The challenge there is that the two cells are separated by nondifferentiability points, which incur a sharp change in the derivative of the function.
The Jackson-type estimate can no longer provide the fast rate of approximation.
We propose to bound the rate of change using quantitative results on the construction of the piecewise-linear approximation, together with Lipschitz continuity of the function we are approximating.
This allows us to bound the distance between the original function, its smooth extension from the first part, and finally its polynomial approximation, which results in the second term of \eqref{eq:them}.

\section{Mixed-integer formulation of piecewise polynomial approximation}
\label{sec:mip-form}

In this section, we formulate the problem of piecewise polynomial regression as a mixed-integer program (MIP), inspired by the optimal classification trees framework \citep{bertsimasOptimalClassificationTrees2017}.


The mixed-integer optimization problem expects as input $\nSample$ points $(x_{i})_{i\in[\nSample]}$ that belong to $\appDom$, and the corresponding function values $y_{i} = f(x_{i})$ for $i\in[\nSample]$.
Without loss of generality, we assume that the sample points belong to $[0, 1]^{\ndim}$.
The output is a piecewise polynomial function; the boundaries of the smooth pieces are defined by affine hyperplanes.
\Cref{table:hyperparam,table:variablesaffhyp} summarize the hyperparameters and variables of the mixed-integer formulation; we now explain the formulation details.


\textbf{Binary tree.} We consider a fixed binary tree of depth $D$.
The tree has $T = 2^{D+1}-1$ nodes, indexed by $t=1, \ldots, T$ such that all branching nodes are indexed by $t=1, \ldots, 2^{D}-1$ and leaf nodes are indexed by $t=2^{D}, \ldots, 2^{D+1}-1$.
The sets of branching nodes and leaf nodes are denoted $\bNode$ and $\lNode$ respectively.
Besides, the set of ancestors of node $t$ is denoted $A(t)$.
This set is partitioned into $A_{L}(t)$ and $A_{R}(t)$, the subsets of ancestors at which branching was performed on the left and right respectively.
\Cref{fig:tree} shows a tree of depth $D = 2$ where
\eg{} the ancestors of node $6$ are $A(6) = \{1, 3\}$, and left and right branching ancestors are $A_{L}=\{3\}$ and $A_{R}=\{1\}$.
Each leaf corresponds to an element of the partition, defined by the inequalities of its ancestors \eg{} leaf $6$ corresponds to $\{x\in\bbR^{\ndim} : a_{1}^{\top}x \ge b_{1}, a_{3}^{\top}x < b_{3}\}$.

\ifthenelse{\boolean{IsArxivVersion}}{
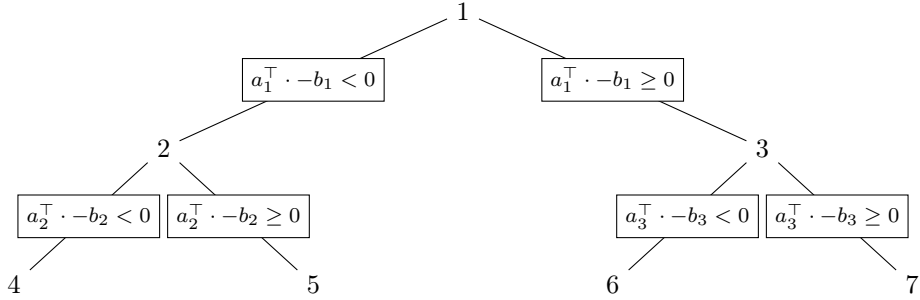
\begin{figure}[t]
        \begin{forest}
          [1 ,for tree={s sep=1.4in,l=12ex},
          [2 ,edge label={node[draw,fill=white,midway]{\footnotesize{$a_{1}^{\top}\cdot - b_{1} < 0$}}}
          [4 ,edge label={node[draw,fill=white,midway] {\footnotesize{$a_{2}^{\top}\cdot - b_{2} < 0$}}}
          ]
          [5 ,edge label={node[draw,fill=white,midway] {\footnotesize{$a_{2}^{\top}\cdot - b_{2} \ge 0$}}}
          ]
          ]
          [3 ,edge label={node[draw,fill=white,midway]{\footnotesize{$a_{1}^{\top}\cdot - b_{1} \ge 0$}}}
          [6 ,edge label={node[draw,fill=white,midway] {\footnotesize{$a_{3}^{\top}\cdot - b_{3} < 0$}}}
          ]
          [7 ,name=P2right,edge label={node[draw,fill=white,midway] {\footnotesize{$a_{3}^{\top}\cdot - b_{3} \ge 0$}}}
          ]
          ]
          ]
        \end{forest}
      \caption{\label{fig:tree}%
        Binary tree and corresponding partition.
  }
\end{figure}
\begin{table}
      \begin{tabular}{ll}
        \toprule
        parameter & interpretation \\ \midrule
        $D$ & depth of the binary tree \\
        $N_{min}$& minimal number of points allowed in a nonempty leaf \\
        $\polyDeg$ & maximum degree of the polynomial on each piece \\
        \bottomrule
      \end{tabular}
      \caption{Summary of the hyperparameters\label{table:hyperparam}}%
\end{table}}{
\begin{figure}[t]
  \begin{floatrow}
    \ffigbox{%
      \resizebox{0.48\textwidth}{!}{%
        \begin{forest}
          [1 ,for tree={s sep=1.4in,l=12ex},
          [2 ,edge label={node[draw,fill=white,midway]{\footnotesize{$a_{1}^{\top}\cdot - b_{1} < 0$}}}
          [4 ,edge label={node[draw,fill=white,midway] {\footnotesize{$a_{2}^{\top}\cdot - b_{2} < 0$}}}
          ]
          [5 ,edge label={node[draw,fill=white,midway] {\footnotesize{$a_{2}^{\top}\cdot - b_{2} \ge 0$}}}
          ]
          ]
          [3 ,edge label={node[draw,fill=white,midway]{\footnotesize{$a_{1}^{\top}\cdot - b_{1} \ge 0$}}}
          [6 ,edge label={node[draw,fill=white,midway] {\footnotesize{$a_{3}^{\top}\cdot - b_{3} < 0$}}}
          ]
          [7 ,name=P2right,edge label={node[draw,fill=white,midway] {\footnotesize{$a_{3}^{\top}\cdot - b_{3} \ge 0$}}}
          ]
          ]
          ]
        \end{forest}
      }
    }{\caption{\label{fig:tree}%
        Binary tree and corresponding partition.
  }}
    \capbtabbox{%
    \resizebox{0.48\textwidth}{!}{%
      \begin{tabular}{ll}
        \toprule
        parameter & interpretation \\ \midrule
        $D$ & depth of the binary tree \\
        $N_{min}$& minimal number of points allowed in a nonempty leaf \\
        $\polyDeg$ & maximum degree of the polynomial on each piece \\
        \bottomrule
      \end{tabular}
      }
    }{%
      \caption{Summary of the hyperparameters\label{table:hyperparam}}%
    }
  \end{floatrow}
\end{figure}
}

\begin{table}[t]
  \caption{Summary of the variables of the affine-hyperplane regression tree formulation\label{table:variablesaffhyp}}
  \begin{center}
  \resizebox{0.6\textwidth}{!}{
    \begin{tabular}{lll}
      \toprule
      variable & index domain & interpretation \\ \midrule
      $l_{t} \in \{0, 1\}$ & $t\in\lNode$ & 1 iff any point is assigned to leaf $t$ \\
      $z_{it} \in \{0, 1\}$ & $t\in\lNode$, $i\in[\nSample]$ & 1 iff point $x_{i}$ is assigned to leaf $t$ \\
      $a_{m} \in \bbR^{\ndim}$ & \multirow{2}{*}{$m\in\bNode$} & \multirow{2}{*}{coefficients of the affine cut} \\
      $b_{m} \in \bbR$ & & \\

      $o_{jm} \in \{0, 1\}$ & $m\in\bNode$, $j\in[\ndim]$ & 1 iff coordinate $j$ of $a_{m}$ is positive \\
      $a_{m}^{+},a_{m}^{-} \in \bbR$ & $m\in\bNode$ & the positive and negative part of $a_{m}$ \\

      $\phi_{it} \in \bbR$ & $t\in\lNode$, $i\in[\nSample]$ & fit error of point $x_{i}$ by the polynomial of leaf $t$ \\
      $\delta_{i} \in \bbR$ & $i\in[\nSample]$ & fit error of point $x_{i}$ by the piecewise polynomial function \\
      $c_{t} \in \bbR^{\binom{\polyDeg + \ndim-1}{\ndim-1}} $ & $t\in\lNode$ & coefficients of the degree $\polyDeg$ polynomial associated with leaf $t$ \\
      \bottomrule
    \end{tabular}
  }
  \end{center}
\end{table}

\textbf{Affine-hyperplane partition of the space.}
At each branching node $m\in\bNode$, a hyperplane splits the space in two subspaces
\begin{equation}\label{eq:splitstrict}
  a_m^{\top} x_i - b_{m}  < 0 \qquad a_m^{\top} x_i - b_{m}  \ge 0,
\end{equation}
that will be associated to the left and right children of node $m$.
The parameters $a_{m}\in\bbR^{\ndim}$ and $b_{m}\in\bbR$ are variables of the mixed integer program.

Making this formulation practical requires two precisions.
First, in order to avoid scaling issues, we constrain $a_{m}$ to belong in $[-1,1]^{\ndim}$, such that $\|a_{m}\|_{1}=1$.\footnote{We depart here from the OCT-H formulation of \cite{bertsimasOptimalClassificationTrees2017}, which constrains the norm of $a_{m}$ to be at most $1$.}
This formulates as
\begin{align*}
  \textstyle\sum_{j=1}^d{(a^+_{jt}+a^-_{jt})} = 1&&  \\
  a_{jt} = a^+_{jt} - a^-_{jt} && \forall j \in [\ndim] \\
  a^+_{jt} \le o_{jt} && \forall j \in [\ndim] \\
  a^-_{jt} \le (1 - o_{jt}) && \forall j \in [\ndim]
\end{align*}
Furthermore, since $x_{i}\in[0, 1]^\ndim$, it holds that $a_{m}^{\top}x_{i}\in[-1, 1]$, so that $b_{m}$ is constrained to $[0, 1]$ without loss of generality.
Second, we implement the strict inequality \eqref{eq:splitstrict} by introducing variable $z_{it}$ 
and a small constant $\mu>0$.
Noting that $a_t^{\top}x_i - b_t$ takes values in the interval $[-2, 2]$, the affine split inequalities \eqref{eq:splitstrict} now formulate as
\begin{align*}
  a_m^{\top} x_i &\ge b_m - 2(1 - z_{it}) & \forall m \in A_R(t)  \\
  a_m^{\top} x_i + \mu &\le b_m + (2 + \mu)(1 - z_{it}) & \forall m \in A_L(t)  
\end{align*}
both for all $i$ in $[\nSample]$ and all $t$ in $\mathcal{T}_L$.

\textbf{Regression on each partition element.}
Each leaf $t\in\lNode$ corresponds to one element of the partition of $\appDom$ defined by the tree, and is associated with a degree $r$ polynomial $\textrm{poly}(\cdot; c_{t})$ whose coefficients are parameterized by variable $c_{t}$.
The regression error of point $x_{i}$ by the polynomial associated with leaf $t$ is then $\phi_{it} = y_i - \textrm{poly}(x_{i}; c_{t})$.
Each point $x_{i}$ is assigned to a unique leaf of the tree: for all $i\in[\nSample]$, $\sum_{t \in \lNode}{z_{it}} = 1$.
Furthermore, each leaf $t\in\lNode$ is either assigned zero or at least $N_{min}$ points: for all $i\in[\nSample]$, $\sum_{i=1}^{\nSample}{z_{it}} \ge N_{\min} l_t$ and $z_{it} \le l_{t}$.

\textbf{Minimizing the regression error.}
The objective is minimizing the average prediction error $\sum_{i=1}^{\nSample}|y_{i} - \textrm{poly}(x_{i}; c_{t(i)})|$, where $c_{t(i)}$ is the polynomial coefficients corresponding to the leaf to which $x_{i}$ belongs.
We formulate this as minimizing $\nSample^{-1} \sum_{i=1}^{\nSample} \delta_{i}$,
where $\delta_{i}$ models the absolute value of the prediction error for point $x_{i}$:
\begin{align*}
  \delta_{i} &\ge \phantom{-}\phi_{it} - (1 - z_{it})M \quad \text{for all } t \in \lNode \\
  \delta_{i} &\ge -\phi_{it} - (1 - z_{it})M \quad \text{for all } t \in \lNode
\end{align*}
where $M$ is a big constant that makes the constraint inactive when $z_{it} = 0$.
Note that the above big-$M$ formulation can be replaced by indicator constraints, if the solver supports them, to encode that if $z_{it} = 1$, then $\delta_i \ge \phi_{it}$,  $\delta_i \ge -\phi_{it}$.
We further note that the formulation could be changed to optimize the mean squared error, by changing the objective function to $\nSample^{-1}\sum_{i=1}^{\nSample} \delta_i^2$.


Combining these elements yields the affine hyperplane formulation, summarized in \cref{eq:OrtFormulationHplane}, \cref{appx:affhyp}.

\begin{remark}[Axis-aligned regression]
  In \Cref{appx:axisalign}, we propose a version of \eqref{eq:OrtFormulationHplane} tailored to functions whose cells have boundaries aligned with cartesian axes.
  Such functions appear in signal processing applications, and are the topic of recent works \citet{chatterjeeAdaptiveEstimationMultivariate2021,bertsimasOptimalClassificationTrees2017,NEURIPS2021_dba4c1a1}.
\end{remark}

\section{Numerical experiments}
\label{sec:numer-exper}

In this section, we demonstrate the applicability of the piecewise polynomial regression method developed in \cref{sec:mip-form}.
\cref{sec:addit-deta-num,sec:numericalexperiments} contains complementary details and experiments.

\paragraph{Setup.}
We implement the affine-hyperplane formulation, detailed in \cref{eq:OrtFormulationHplane}, in Python and use Gurobi as the MIP solver.
We present three applications: regression of the cone function \eqref{eq:cone2dintro}, regression of a Neural Network, and denoising of a piecewise constant 2d signal, following the experiments of \citep{NEURIPS2021_dba4c1a1}.
We use trees of varying depth $D \in \{2,3,4\}$ and polynomial degree $\polyDeg \in \{0,1,2\}$.
We set $N_{\min}=1$ and $\mu=10^{-4}$ across all experiments, and run experiments on a personal laptop.


\subsection{Regression of tame functions}
\label{sec:regr-tame-funct}

We consider the regression problem for two tame functions and report in \cref{fig:regreexpes} the level lines of the original function and the level lines of the piecewise polynomial approximations provided by the affine-hyperplane formulation \eqref{eq:OrtFormulationHplane}.
\Cref{table:expsregression} shows the error computed on a $1000 \times 1000$ grid of regularly spaced points. We divide the absolute errors by the maximal absolute value of the underlying ground truth.



\paragraph{Piecewise-affine function.}
Firstly, we return to the ``cone'' function. 
This function is piecewise-linear and continuous; see \cref{fig:cone} (left pane) for an illustration.
This is a nonconvex function for which estimating the correct full-dimensional cells with a sampling scheme gets harder as $\rCone$ goes to zero or as the dimension of the space $\ndim$ increases.

\Cref{fig:cone} (middle and right pane) shows the obtained piecewise approximation of $\fCone$ (with $\rCone=\sCone=0.5$), for depth $2$ and $3$ approximation trees with time budget of 5 and 10 minutes.
There are $\nSample = 250$ points sampled uniformly; the polynomial degree is $\polyDeg=1$.
The depth $D=2$ tree fails to recover an approximation of the cells, as it can only approximate $4$ cells while $\fCone$ has $7$ cells.
The depth $D = 3$ recovers the cell decomposition well qualitatively, and reduces the error measures by a factor $2$; see \cref{table:expsregression}.

\begin{figure}[t]
  \centering
  \begin{subfigure}{0.48\columnwidth}
    \includegraphics[width=\textwidth]{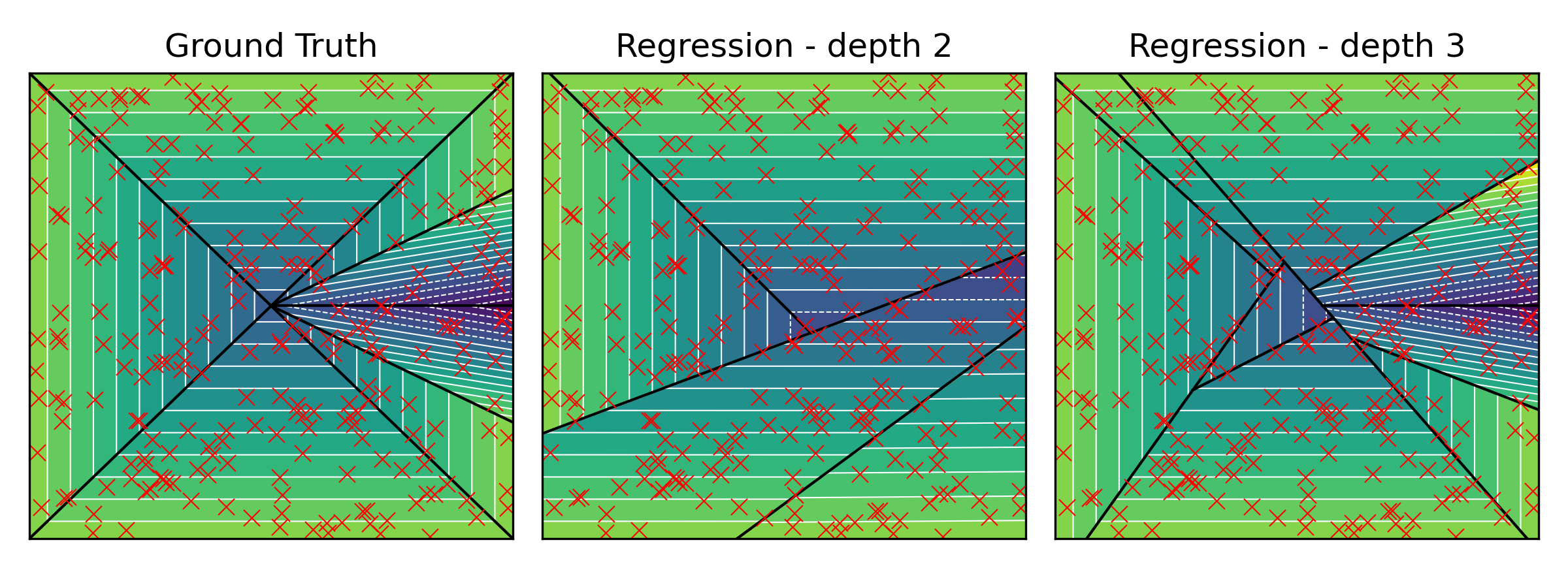}
    \caption{The cone function \eqref{eq:cone2dintro}.}
    \label{fig:cone}
  \end{subfigure}
  \begin{subfigure}{0.48\columnwidth}
    \includegraphics[width=\textwidth]{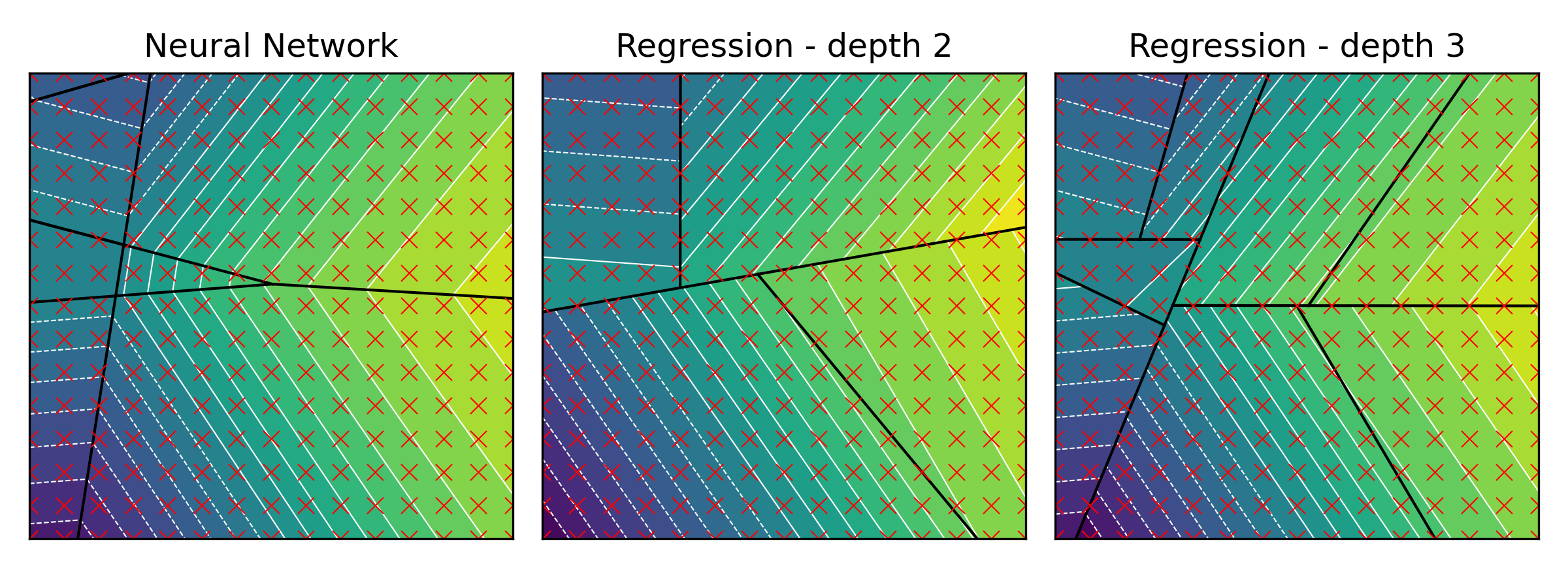}
    \caption{The neural network from \cref{fig:NN_architecture}}
    \label{fig:NN_fit}
  \end{subfigure}
  \caption{Tame functions and regression results for trees of depth 2 and 3. Red crosses are the coordinates of samples used for the training. Black lines show the decomposition of the space.\label{fig:regreexpes}}
\end{figure}



\paragraph{Non-piecewise-linear Neural Network}

We consider a small neural network, comprised of 27 parameters, sigmoid, $\tanh$, ReLU activations, and 1d max-pooling; see the architecture in \Cref{fig:NN_architecture}.

The neural network is trained to approximate the 2d function $x\mapsto 2\sin{x_1} + 2\cos{x_2} + x_2 / 2$, from 15 random samples on the cube $[-2,2]^{2}$.
The loss is the mean squared error over the 15 samples; it is optimized in a single batch for 5000 epochs using the \texttt{AdamW} algorithm.


We approximate the output of the trained NN by taking $\nSample = 225 = 15^{2}$ points in a $15\times15$ regular grid over the approximation space and optimize the trees using that as input.
We set the polynomial degree $\polyDeg=2$.


\Cref{fig:NN_fit} (middle and right pane) shows the obtained piecewise approximation of the Neural Network, for depth $2$ and $3$ approximation trees with time budget of 30 and 60 minutes. 
Increasing the depth from 2 to 3 reduces again the error measures by a factor of 2; see \Cref{table:expsregression}. 

\subsection{Denoising of a piecewise-constant 2d signal}
\label{sec:deno-piec-const}

\begin{figure}[t]
  \centering
  \includegraphics[width=0.5\columnwidth]{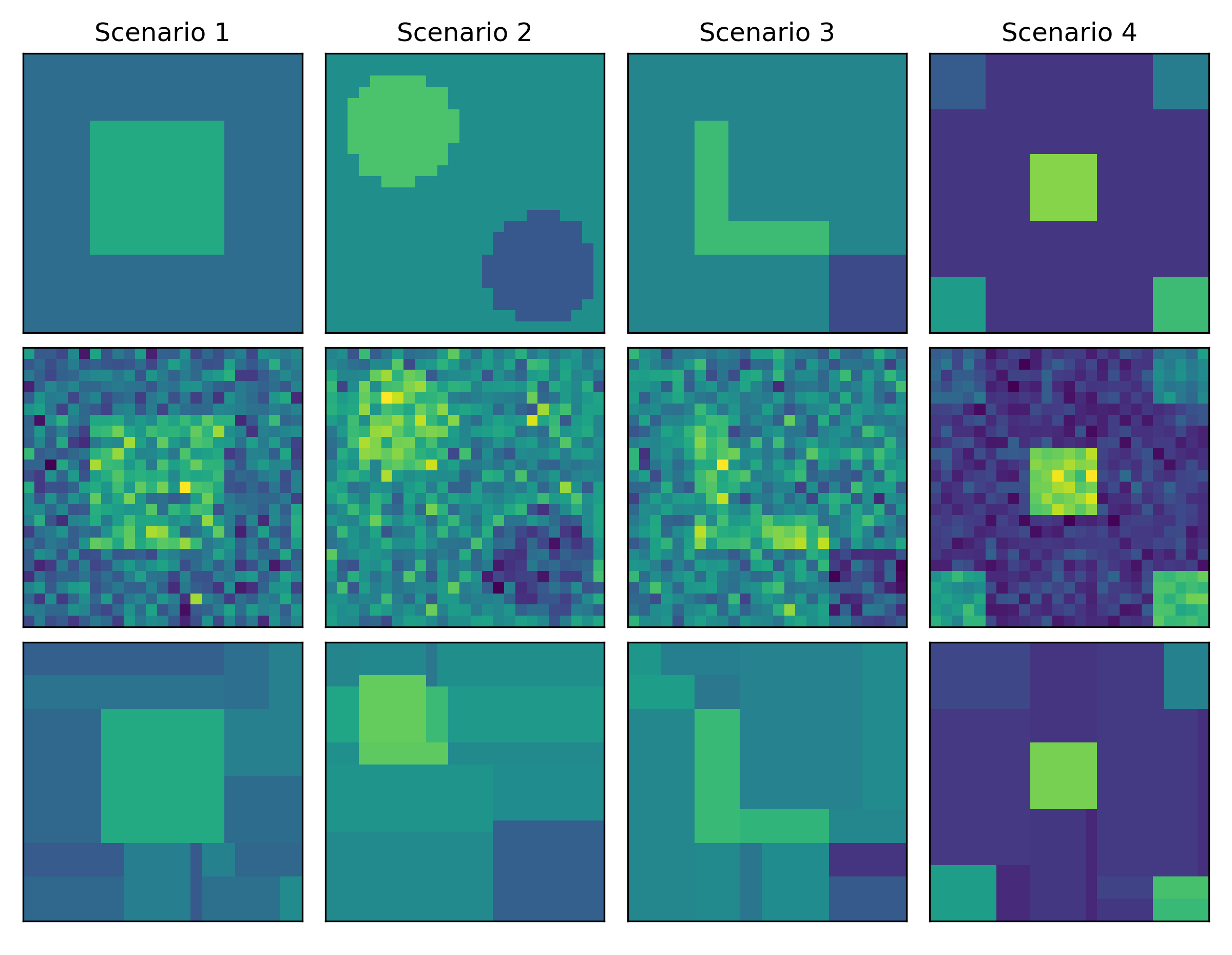}
  \caption{
    Four denoising scenarios from \citet[Figure 3]{NEURIPS2021_dba4c1a1}.
    The regression trees are restricted to axis-aligned splits and have depth 4.
    First row: ground truth, second row: (corrupted) regression signal, third row: recovered signal from the mixed integer formulation.}
  \label{fig:data_noisy}
\end{figure}

We consider now a slightly different set-up: following \citet{NEURIPS2021_dba4c1a1}, we consider four two-dimensional functions that are piecewise constant on a 25 by 25 grid, illustrated in \Cref{fig:data_noisy} (row one).
The regression data is the function corrupted by an additive Gaussian with zero mean and standard deviation of $\sigma = 0.5$, illustrated in \Cref{fig:data_noisy} (row two).
Row three of \Cref{fig:data_noisy} shows, for each scenario, the output of regression trees with depth $D=4$ and polynomial degree $\polyDeg=0$, with a time limit of 1 hour.

The piecewise polynomial formulation used here is a simplified version of \Cref{eq:OrtFormulationHplane}, where the splits are restricted to align with the Cartesian axes; see \cref{appx:axisalign} for details.
\Cref{table:expsdenoising} shows the recovery error of the regression trees; the recovered signals are comparable in quality to Figure 3 in \citet{NEURIPS2021_dba4c1a1}.

\subsection{Notes on the optimization process}
\label{sec:note-optim-process}
To study the MIP optimization process until completion, we choose the $\norm{\cdot}_{\infty}$ norm.
This is a simple example of a tame function with partitions unfit for approximation by trees with axis-aligned splits, as in \eg{} the recent works of \citet{NEURIPS2021_dba4c1a1,chatterjeeAdaptiveEstimationMultivariate2021}.

We optimize over $\nSample=100$ randomly sampled points with depth $D=2$ and polynomial degree $\polyDeg=1$. 
In the Appendix,  \Cref{fig:mip_process} shows the evolution of objective value bounds: the blue curve shows the objective value of the best candidate found by the solver so far, the orange curve shows the best lower bound on the optimal objective value found so far.
A candidate piecewise polynomial function is proved to be optimal when the two curves coincide.
Here, the mixed-integer solver takes almost 50 minutes to find a candidate piecewise polynomial function \emph{and} prove its optimality.
But already after 11 seconds, the candidate function has an objective value within $10^{-6}$ of the optimal objective value.
Most of the effort in solving \eqref{eq:OrtFormulationHplane} is thus spent on proving optimality of the candidate function.
This is an even more extreme contrast between the optimization of the two bounds as compared to the results for OCT \cite{bertsimasOptimalClassificationTrees2017}.
This suggests that although the rest of our results were not proven optimal, the solutions could be close to optimal after a short computing time.

\section{Conclusions and limitations}

Our numerical results showcase a proof-of-concept implementation, whose scalability (in terms of global optimality) is limited to low-dimensional functions, while a feasible solution can be found surprisingly quickly. Cf. \Cref{fig:mip_process} in the Appendix.  
Improvements in the search for lower (dual) bounds is an important direction of future work.



\appendix
\setcounter{remark}{0}
\numberwithin{equation}{section}		
\numberwithin{lemma}{section}		
\numberwithin{proposition}{section}		
\numberwithin{theorem}{section}		
\numberwithin{corollary}{section}		

\section{Tame geometry}
\label{appx:def}

We give a more formal but brief review of the results needed from tame, or o-minimal, geometry.
We refer the interested reader to \citet{van1998tame,costeIntroductionOminimalGeometry2000} for extensive expositions.

Let us start with the definition of an o-minimal structure.
For simplicity, we consider only structures over the real field $\bbR$.
\begin{definition}[o-minimal structure]\label{def:omin}
  An o-minimal structure on $\bbR$ is a sequence $\CS=(\CS_m)_{m\in\bbN}$ such that for each $m\geq 1$:
  \begin{enumerate}[i.]
    \item $\CS_m$ is a boolean algebra of subsets of $\bbR^m$;
    \item if $A\in\CS_m$, then $\bbR\times A$ and $A\times \bbR$ belongs to $\CS_{m+1}$;
    \item $\{(x_1,\dots,x_m)\in \bbR^m\,:\, x_1=x_m\}\in \CS_m$;
    \item if $A\in\CS_{m+1}$, and $\pi:\,\bbR^{m+1}\to\bbR^m$ is the projection map on the first $m$ coordinates, then $\pi(A)\in\CS_m$;
    \item the sets in $\CS_1$ are exactly the finite unions of intervals and points.
  \end{enumerate}
\end{definition}

A set $A\subseteq \bbR^m$ is said to be \emph{definable} in $\CS$ if $A$ belongs to $\CS_m$.
Similarly, a map $f:\, A\to B$, with $A\subseteq \bbR^m$, $B\subseteq \bbR^n$, is said to be definable in $\CS$ if its graph $\Gamma(f) \eqdef \{(x,f(x))\in\bbR^{m+n}: x\in A\}$ belongs to $\CS_{m+n}$.

A set or function definable in some o-minimal structure is often simply referred to as \emph{tame} when the specific o-minimal structure is not important.

The most fundamental o-minimal structure is that containing the semialgebraic sets:
\begin{example}[Semialgebraic structure]
  A \emph{semialgebraic set} in $\bbR^{\ndim}$ defines as a finite union of sets of the form
  \begin{equation}\label{eq:semi_alg}
    \{ x \in \bbR^{\ndim} : f_1(x) = 0,\ldots, f_k(x)=0, g_{1}(x) > 0, \ldots, g_{l}(x) > 0\}
  \end{equation}
  where the $f_i$ and the $g_i$ are all real polynomials in $\ndim$ variables.

  The collection of all semialgebraic sets forms an o-minimal structure. In particular, any function built from polynomials, boolean rules, and coordinate projections is definable in that structure. Examples include affine applications, piecewise polynomial functions such as linear layers, (pointwise) ReLU activation functions, and their composition, or functions such as $(x, y) \mapsto \sqrt{x^{2} + y^{3}}$.
\end{example}

As stated in \cref{prop:celldecomp}, tame sets can always be partitioned into smaller subsets.
For completeness, we state the formal cell decomposition theorem.
To this end, we first need the following.

\begin{definition}[Cells, decompositions]\label{def:celldecomp}
  We define both cells and decompositions inductively. A \emph{cell} in $\bbR$ is a point $\{a\}$, for $a\in\bbR$ or an (open) interval $(a,b)$, for $a,b\in\bbR\cup\{\pm\infty\}$.
  \begin{itemize}
    \item Let $C$ be a cell in $\bbR^m$ and let $f:\, C\to\bbR$ be tame and continuous, then $\{(x,f(x))\,:\,x\in C\}$ is a cell in $\bbR^{m+1}$.
    \item Let $C$ be a cell in $\bbR^m$ and let $f,g:\, C\to\bbR\cup\{\pm\infty\}$ be tame and continuous such that for all $x\in C:\,f(x)<g(x)$, then $\{(x,y)\in C\times \bbR\,:\, f(x)<y<g(x)\}$ is a cell in $\bbR^{m+1}$.
  \end{itemize}

  A \emph{decomposition} of $\bbR$ is a finite partition of the form
  \begin{equation}
    \{(-\infty,a_1),(a_1,a_2),\dots, (a_n,+\infty),\{a_1\},\dots,\{a_n\}\},
  \end{equation}
  and a decomposition of $\bbR^{m+1}$ is a finite partition of $\bbR^{m+1}$ into cells $C_1,\dots,C_n$ such that the set of projections $\pi(C_j)$ gives a partition of $\bbR^m$.
\end{definition}
Cells are connected sets.
We further say that a cell is $\mathcal{C}^\smoothDeg$ if all functions used to construct the cell are $\mathcal{C}^\smoothDeg$. We now have the fundamental result:

\begin{theorem}[Cell decomposition, cf. Thm. 7.3.2 \cite{van1998tame}]
For any definable sets $A,\, A_1,\dots, A_n\subseteq \bbR^m$ and definable function $f:A\to\bbR$, there is
    \begin{itemize}
        \item a decomposition of $\bbR^m$ into $\C^\smoothDeg$-cells partitioning the sets $A_j$.
        \item a decomposition of $\bbR^m$ into $\C^\smoothDeg$-cells partitioning $A$, such that each restriction $f|_C:C\to\bbR$ is $\C^\smoothDeg$ for each cell $C\subseteq A$ of the decomposition.
    \end{itemize}
\end{theorem}

\begin{example}[\cref{ex:conecelldecomp} continued]
  We return on the cone function, defined in \cref{eq:cone2dintro} as
  \begin{equation}
    f(x) =
    \begin{cases}
      -\sCone x_{1} + \frac{1+\sCone}{\rCone}x_{2} & \text{ if } x_{1} > 0 \text{ and } 0 < x_{2} < \rCone x_{1} \\
      -\sCone x_{1} - \frac{1+\sCone}{\rCone}x_{2} & \text{ if } x_{1} > 0 \text{ and } -\rCone x_1 < x_{2} < 0 \\
      \|x\|_{\infty} & \text{ else}
    \end{cases}.
  \end{equation}
  The above cell decomposition theorem provides the cells described in \cref{table:conecells}; see \cref{fig:conelevelsjointappx} for an illustration.
\end{example}

\begin{figure}[t]
  \begin{center}
    \includegraphics[width=0.48\textwidth]{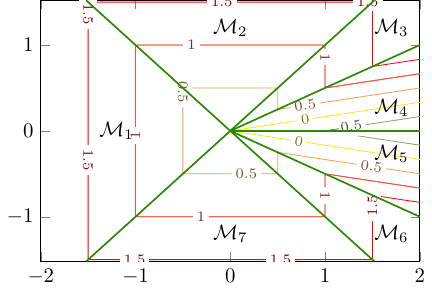}
    \caption{Illustration of the ``cone'' function \eqref{eq:cone2dintro}, with $\sCone=\rCone=0.5$, showing \emph{(i)} the level lines of the function, and \emph{(ii)} the decomposition of the domain into cells on which the function is smooth, as provided by \cref{prop:celldecomp}; see \cref{table:conecells} for details.}
    \label{fig:conelevelsjointappx}
  \end{center}
  \vskip -0.2in
\end{figure}
\begin{table}[t]
  \caption{Cell decomposition of the ``cone'' function \eqref{eq:cone2dintro} with $\rCone=\sCone=0.5$: the space decomposes in $7$ open full-dimensional sets on which the function is smooth; see \cref{fig:conelevelsjoint} for an illustration.\label{table:conecells}}
  \centering
  \begin{tabular}{lll}
    \toprule
    Cell & cell expression & $f$ restricted to cell \\ \midrule
    $\M_{1}$ & $\{x\in\bbR^{2}: x_{1}+x_{2}<0, \; x_{1}-x_{2}<0\}$ & $f|_{\M_{1}} = -x_{1}$ \\
    $\M_{2}$ & $\{x\in\bbR^{2}: x_{1}+x_{2}<0, \; x_{1}-x_{2}<0\}$ & $f|_{\M_{2}} = \phantom{-}x_{2}$ \\
    $\M_{3}$ & $\{x\in\bbR^{2}: x_{1}+x_{2}>0, \; x_{1}-x_{2}>0\}$ & $f|_{\M_{3}} = \phantom{-}x_{1}$ \\
    $\M_{4}$ & $\{x\in\bbR^{2}: x_{1}>0, \; 0 < \phantom{-}x_{2}< 0.5x_{1}\}$ & $f|_{\M_{4}} = -0.5x_{1} + 3x_{2}$ \\
    $\M_{5}$ & $\{x\in\bbR^{2}: x_{1}>0, \; 0 < -x_{2}< 0.5x_{1}\}$ & $f|_{\M_{5}} = -0.5x_{1} - 3x_{2}$ \\
    $\M_{6}$ & $\{x\in\bbR^{2}: x_{1}+x_{2}>0, \; x_{1}-x_{2}>0\}$ & $f|_{\M_{6}} = \phantom{-}x_{1}$ \\
    $\M_{7}$ & $\{x\in\bbR^{2}: x_{1}+x_{2}<0, \; x_{1}-x_{2}<0\}$ & $f|_{\M_{7}} = -x_{2}$ \\
    \bottomrule
  \end{tabular}
\end{table}

A related notion to that of cell decomposition is that of stratifications. A stratification is slightly stronger than a cell decomposition, in that it gives further conditions on how the different pieces fit together. The basic idea is that the set is partitioned into a finite number of manifolds, called the \emph{strata}, with some additional conditions on how the different strata glue together. Various types of stratifications, with differing conditions on the gluing of the strata, exist in the literature. Some examples are the Whitney, Verdier and Lipschitz conditions; see \eg{} \citet{le1998verdier}. For o-minimal structures, we can again require that the function is $\C^\smoothDeg$, for some $\smoothDeg<\infty$, on each strata. The additional gluing conditions have played a vital role in the recent optimization literature when addressing \eg{} questions of convergence for gradient descent algorithms to Clarke critical points \cite{davisStochasticSubgradientMethod2020, bolte2021conservative, davis2021active}.



\section{Proof of \cref{th:main}}
\label{sec:proof-th}
In this Section, we provide the proof of \cref{th:main}.

\paragraph{Notation}
We follow the standard notion of the differential $\D^{\gamma} f$ of a function $f:\inputSpace\to\bbR$, for a multi-index $\gamma\in\inputSpace$.

Following the setup of \cite{feffermanExtensionLinearOperators2007}, we denote by $\Cm(\inputSpace)$ the space of real-valued functions on $\inputSpace$ with continuous and bounded derivatives through order $\smoothDeg$; and
\begin{equation}
  \|F\|_{\Cm(\inputSpace)} = \max_{|\gamma|\le m} \sup_{x\in\bbR^{n}} |\D^{\gamma} F(x)|.
\end{equation}
Furthermore, for a subset $E$ of $\inputSpace$, we denote by $\Cm(E)$ the Banach space of all real-valued functions $\varphi$ on $E$ such that $\varphi=F$ on $E$ for some $F\in\Cm(\inputSpace)$.
The natural norm on $\Cm(E)$ is given by
\begin{align}
  \|F\|_{\Cm(E)} &= \inf_{\varphi\in\Cm(\inputSpace) \text{ and } \varphi = F \text{ on } E} \|\varphi\|_{\Cm(\inputSpace)}.
\end{align}
We collect here a simple remark:
\begin{lemma}\label{lmm:banachnorm}
  Consider two subsets $E$ and $F$ of $\inputSpace$.
  If $E\subset F$, then for any $f\in\Cm(F)$, there holds
  \begin{equation}
    \|f\|_{\Cm(E)} \le \| f \|_{\Cm(F)}.
  \end{equation}
\end{lemma}
\begin{proof}
The property follows from the fact that, since $E\subset F$, the set of functions $\varphi\in\Cm(\inputSpace)$ such that $\varphi=f$ on $F$ is contained in the set of functions $\varphi\in\Cm(\inputSpace)$ such that $\varphi = f$ on $E$.
\end{proof}

Finally, we need the following notion of distance between a pair of manifolds:
\begin{definition}
    The Frechet distance $d_F(\cdot,\cdot)$ between two homeomorphic manifolds is defined as
    \begin{equation*}
        d_F(\mathcal{M}_a,\mathcal{M}_b)=\inf_{h\in \mathcal{H}}\sup_{x\in\mathcal{M}_a}d(x,h(x)),
    \end{equation*}
    where $\mathcal{H}$ is the set of all homemorphisms from $\mathcal{M}_a$ to $\mathcal{M}_b$. 
\end{definition}

\subsection{A Jackson and extension theorem for subsets of $\cancube$}
We provide a variant of the Jackson theorem that allows for fast approximation of smooth functions.
The main difference relative to classical statements \eg{} \cite{plesniakMultivariateJacksonInequality2009,bagbyMultivariateSimultaneousApproximation2002}, is that the constant $\Cnm$ does not depend on the domain $E\subset\cancube$ where the function is approximated, only on the working domain $\cancube$.

\begin{lemma}\label{lmm:Jackson}
  Fix $\ndim$ and $\smoothDeg$ positive integers and assume that $\appDom$ is a connected compact subset of $\inputSpace$, such that any two points $x$ and $y$ in $\appDom$ can be joined by a rectifiable arc in $\appDom$ with length no greater that $\sigma\|x-y\|$, where $\sigma$ is a positive constant.
  
  Then, there exists a constant $\Cfeff$ that depends only on $\ndim$ and $\smoothDeg$, and a constant $\Cnm$ that depends only on $\ndim$, $\smoothDeg$, and $\cancube$ such that the following holds: for any subset $E$ of $\cancube$,
  \begin{enumerate}[i.]
      \item there exists a linear operator $T:\Cm(E)\to\Cm(\inputSpace)$ such that $T$ extends functions of $E$ to $\inputSpace$, and the operator norm of $T$ is bounded by $\Cfeff$, so that for any $f\in\Cm(E)$,
          \begin{equation}\label{eq:Feffext}
            \|T(f)\|_{\Cm(\inputSpace)} \le \Cfeff \|f\|_{\Cm(E)}.
          \end{equation}
      \item for any function $f\in\C^{\smoothDeg}(E)$, and integer $\polyDeg$, there exists a polynomial $p_{\polyDeg}$ of degree at most $\polyDeg$ such that
  \begin{equation}
    \| f - p_{\polyDeg}\|_{\infty, E} \le \Cnm \frac{1}{\polyDeg^{\smoothDeg}} \|f\|_{\Cm(E)}.
  \end{equation}
  \end{enumerate}
\end{lemma}

\begin{proof}
  The first item is exactly \Citet[Th. 1]{feffermanExtensionLinearOperators2007}.
  For the second item, applying Jackson's theorem \cite[Th. 2]{bagbyMultivariateSimultaneousApproximation2002} to $T(f)$ on $\cancube$ provides the existence of a polynomial $p_{\polyDeg}$ of degree up to $\polyDeg$ such that:
  \begin{equation}\label{eq:Jackson}
    \|T(f) - p_{\polyDeg}\|_{\infty, \cancube} \le \Cjack \frac{1}{\polyDeg^{\smoothDeg}} \sum_{|\gamma|\le \smoothDeg}|\D^{\gamma} T(f)|_{\infty, \cancube}.
  \end{equation}
  Note that the assumptions of Theorem 2 \cite{bagbyMultivariateSimultaneousApproximation2002} are verified: $T(f)$ is indeed of class $\C^{\smoothDeg}$ on a neighborhood of $\cancube$.

Finally, note that
  \begin{multline}\label{eq:normbound}
    \sum_{|\gamma|\le \smoothDeg}|\D^{\gamma} T(f)|_{\infty, \cancube} \le \sum_{|\gamma|\le \smoothDeg}|\D^{\gamma} T(f)|_{\infty, \inputSpace}  \\ \le \binom{\smoothDeg+\ndim}{\ndim} \max_{|\gamma|\le \smoothDeg}|\D^{\gamma} T(f)|_{\infty, \inputSpace} = \binom{\smoothDeg+\ndim}{\ndim} \| T(f) \|_{\Cm(\inputSpace)}.
  \end{multline}

  The above elements combine as follows,
  \begin{align*}
    \| f - p_{\polyDeg}\|_{\infty, E}
    &= \| T(f) - p_{\polyDeg}\|_{\infty, E} \\
    &\le \| T(f) - p_{\polyDeg}\|_{\infty, \cancube} \\
    &\overset{\eqref{eq:Jackson}}{\le} \Cjack \frac{1}{\polyDeg^{\smoothDeg}} \sum_{|\gamma|\le \smoothDeg}|\D^{\gamma} T(f)|_{\infty, \cancube} \\
    &\overset{\eqref{eq:normbound}}{\le} \Cjack \binom{\smoothDeg+\ndim}{\ndim} \frac{1}{\polyDeg^{\smoothDeg}} \|T(f)\|_{\Cm(\inputSpace)}  \\
    &\overset{\eqref{eq:Feffext}}{\le} \Cjack \binom{\smoothDeg+\ndim}{\ndim}  \Cfeff \frac{1}{\polyDeg^{\smoothDeg}} \|f\|_{\Cm(E)},
  \end{align*}
  which shows the claim.
\end{proof}

\subsection{Piecewise-linear approximation of the cell decomposition}

\begin{assumption}[On the cell-decomposition]\label{assump:stratification}
    We consider a cell decomposition $\Mcol$, as per \cref{def:celldecomp}, such that each cell $\M$ is a manifold with boundaries. This is typically referred to as a stratification, see \cref{appx:def}. The stratification $\Mcol$ of $\cancube$ is such that for every codimension $1$ strata $\M$ there exists an application $\Phi_{\M}:\inputSpace\to\bbR$ such that:
    \begin{itemize}
        \item zero is a regular point of $\Phi_\M$, and the zero-set of $\Phi_\M$ defines $\M$: $\M = \Phi_\M^{-1}(\{0\})$; 
        \item $\Phi_\M$ is $\C^2(\cancube)$ and its gradient and Hessian are non-zero and bounded away from infinity.
    \end{itemize}
\end{assumption}

\begin{proposition}\label{prop:strattobintree}
  Consider a stratification $\Mcol$ definable in an o-minimal expansion of $\bbR$ that meets \cref{assump:stratification}. There exists a binary tree partition of $\cancube$ with depth $\nCuts$, as defined in \cref{def:piecepolycuts}, such that the distance between each codimension 1 strata, $\M$, and the corresponding partition boundary, $\widehat{\M}$, is bounded by 
  \begin{equation}
      d_F(\M, \widehat{\M}) \le \Cstrat \nCuts^{-\frac{2}{\ndim-1}},
  \end{equation}
  where the constant $\Cstrat$ depends only on the dimension $\ndim$ and the geometry of the strata $\M$. 
\end{proposition}
\begin{proof}
We apply the algorithm of \cite{boissonnatTracingIsomanifoldsTime2021} to the closure of the strata $\M$.
From \cite[Thm. 3.4]{boissonnatTracingIsomanifoldsTime2021} we have that $\dist(\M,\widehat{\M}) \le \Cdist D^2 $, where $D$ is the maximal diameter of a linear piece. Here the constant depends on the magnitude of the gradient and Hessian of the mapping $\Phi_\M$ from Assumption \ref{assump:stratification}. From \cite[Prop. 3.6]{boissonnatTracingIsomanifoldsTime2021} we have $\nCuts\leq \Cerr D^{-(\ndim-1)}$, where now the constant $\Cerr$ depends on the space dimension $\ndim$, and the number of times any straight line intersects $\M$. Note that this number is finite by the definability assumption, as is well-known in the case of algebraic manifolds. Putting this together gives the result. 
\end{proof}


\subsection{Proof of the main result}

We are now ready to prove the main result \cref{th:main}, which we restate here for convenience.
\mainres*


\begin{proof}
  The proof consists in constructing a piecewise polynomial function in $\piecePolySpaceCuts$ that has the claimed distance to $f$.

  We first construct the depth-$\nCuts$ binary tree that defines the pieces of the piecewise polynomial.
  Since $f$ is definable in an o-minimal structure, \cref{prop:celldecomp} yields a $\Cm$-decomposition of $\cancube$ such that
  \begin{itemize}
      \item each cell is a connected $\Cm$-manifold,
      \item the restriction of $f$ to each cell $\M$ is $\Cm(\M)$.
  \end{itemize}
  If the decomposition meets \cref{assump:stratification}, \cref{prop:strattobintree} provides a way to recursively split the space along affine hyperplanes such that, for any cell, or stratum, $\M$ that is not full-dimensional and any point $x$ in $\M$, there exists a point of the boundary between the pieces that is at most $\epsmarg$ away from $x$.
  This specifies exactly the geometry of the pieces of the piecewise polynomial function we construct.

  We now turn to define the polynomial on an arbitrary piece $\piece$ on which the piecewise polynomial function is a polynomial.
  Note that by construction of the geometry of the pieces, there can be at most one (full-dimensional) cell such that there exists a point both in $\piece$ and the cell that are more that $\varepsilon$ away from the boundaries of the cell, that is, the points where $f$ is not $\Cm$.
  We let $\Mbig$ denote such a cell, if it exists, or otherwise, an arbitrary cell that intersects with the piece.

  Consider a point $x$ in $\piece$.
  \Cref{lmm:Jackson} provides \emph{(i)} $T(f)$, a smooth extension of $f$ from $\piece\cap\Mbig$ to a neighborhood of $\cancube$, the norm of which is bounded by $\Cfeff$ according to \eqref{eq:Feffext}, and \emph{(ii)} $p_{\polyDeg}^\piece$, the polynomial of degree at most $\polyDeg$ such that, for all $x\in \cancube$,
  \begin{equation}\label{eq:bounda}
    |T(f)(x) - p^{\piece}_\polyDeg(x)| \le \Cnm \frac{1}{\polyDeg^{\smoothDeg}} \|f\|_{\Cm(\piece\cap \Mbig)}.
  \end{equation}

  The triangular inequality yields
  \begin{equation}\label{eq:trianga}
      |f(x) - p^{\piece}_\polyDeg(x)| \le |f(x) - T(f)(x)| + |T(f)(x) - p_\polyDeg(x)|.
  \end{equation}
  The second term of \eqref{eq:trianga} is readily bounded by \eqref{eq:bounda}.

  We now turn to bound the first term of \eqref{eq:trianga}.
  We have, by the triangular inequality
  \begin{multline}
      |f(x) - T(f)(x)| \le |f(x) - f(\projM(x))| \\+ |f(\projM(x))-T(f)(\projM(x))| + |T(f)(\projM(x)) - T(f)(x)|.
  \end{multline}
  First, since $f$ is $K$-Lipschitz,
  \begin{equation}\label{eq:boundb}
    |f(x) - f(\projM(x))| \le K \|x - \projM(x)\|.
  \end{equation}
  Second, since $T(f)$ is an extension of $f$ on $\piece\cap\Mbig$, and $\projM(x)$ belongs to the closure of that space, there holds $f(\projM(x))=T(f)(\projM(x))$.
  Third, since $T(f)$ is $\smoothDeg$-times continuously differentiable on $\inputSpace$,
  \begin{equation}\label{eq:boundc}
    |T(f)(\projM(x)) - T(f)(x)| \le \sup_{y\in\inputSpace} \|\nabla T(f)(y)\| \|x - \projM(x)\|.
  \end{equation}
  Combining the inequality $\|\cdot\|_{2} \le \ndim\|\cdot\|_{\infty}$ and the definition of the norm on the Banach space $\Cm(\inputSpace)$,
  \begin{equation}\label{eq:boundd}
    \sup_{y\in\inputSpace} \|\nabla T(f)(y)\| \le  \ndim \|T(f)\|_{\Cm(\inputSpace)} |.
  \end{equation}

  Combining the above bounds \cref{eq:bounda,eq:boundb,eq:boundc,eq:boundd} yields
  \begin{equation}
      |f(x) - p_{\polyDeg}^\piece(x)| \le (K + \ndim\|T(f)\|_{\Cm(\inputSpace)}) \|x - \projM(x)\| + \Cnm \frac{1}{\polyDeg^{\smoothDeg}} \|f\|_{\Cm(\piece\cap \Mbig)}.
  \end{equation}

  Now, since $\|x-\projM(x)\| \le \epsmarg$ for $x\in\piece$, and using the bound on the norm of $T$ \eqref{eq:Feffext} yields, for all $x\in\piece$,
  \begin{equation}\label{eq:abovething}
      |f(x) - p_{\polyDeg}^\piece(x)| \le (K + \ndim \Cfeff \|f\|_{\Cm(\piece\cap\Mbig)}) \epsmarg + \Cnm \frac{1}{\polyDeg^{\smoothDeg}} \|f\|_{\Cm(\piece\cap \Mbig)}.
  \end{equation}
  We can now use \cref{lmm:banachnorm} to deduce that
  \begin{equation}
    \|f\|_{\Cm(\piece\cap\Mbig)} \le \|f\|_{\Cm(\Mbig)} \le \max_{\M\in\Mcol \text{ such that } \dim(\M) = \ndim} \|f\|_{\Cm(\M)} \eqdef  \Cfbanach,
  \end{equation}
  so that the bound of \eqref{eq:abovething} is now independent of $\piece$:
  \begin{equation}
    |f(x) - p_{\polyDeg}^\piece(x)| \le (K + \ndim \Cfeff \Cfbanach) \epsmarg + \Cnm \frac{1}{\polyDeg^{\smoothDeg}} \Cfbanach.
  \end{equation}

  Taking the supremum over all $x$ in $\piece$, and then over all pieces $\piece$ of $\appDom$ yields the bound for the piecewise polynomial function $p$
  \begin{equation}
      \| f - p \|_{\cancube, \infty} \le C_1 \polyDeg^{-\smoothDeg} + C_2 \nCuts^{-\frac{2}{\ndim-1}}.
  \end{equation}
  where $C_1 = \Cnm \Cfbanach$ depends only on $\ndim$, $\smoothDeg$, $\cancube$, and $f$, and $C_2 =(K + \ndim \Cfeff \Cfbanach) \Cstrat$ depends only on $\ndim$, $\smoothDeg$, and $f$.
\end{proof}




\section{Affine-hyperplane regression formulation: complete formulation}
\label{appx:affhyp}
In this section, we provide the complete affine-hyperplane regression formulation, introduced in \Cref{sec:mip-form}.

\begin{subequations}
  \label{eq:OrtFormulationHplane}
  \begin{align}
    \min \; & \frac{1}{\nSample} \sum^{\nSample}_{i=1} \delta_i & \\
    \textrm{s.t.} \quad & \delta_i \ge \phi_{it} - (1 - z_{it})M && \forall i \in [\nSample], \quad \forall t \in \mathcal{T}_L \\
            & \delta_i \ge - \phi_{it} - (1 - z_{it})M && \forall i \in [\nSample], \quad \forall t \in \mathcal{T}_L \\
            & \phi_{it} = y_i - \textrm{poly}(x_i; c_t) && \forall i \in [\nSample], \quad \forall t \in \mathcal{T}_L \\
            & a_m^{\top} x_i \ge b_m - 2(1 - z_{it}) && \forall i \in [\nSample], \quad \forall t \in \mathcal{T}_L, \quad \forall m \in A_R(t) \label{eq:eee}\\
            & a_m^{\top} x_i + \mu \le b_m & & \notag \\
            & \qquad + (2 + \mu)(1 - z_{it})& & \forall i \in [\nSample], \quad \forall t \in \mathcal{T}_L, \quad \forall m \in A_L(t) \label{eq:fff}\\
            & \textstyle\sum_{t \in \mathcal{T}_L}{z_{it}} = 1 && \forall i \in [\nSample]  \label{eq:ggg}\\
            & z_{it} \le l_t && \forall i \in [\nSample], \quad \forall t \in \mathcal{T}_L \label{eq:hhh}\\
            & \textstyle\sum_{i=1}^{\nSample}{z_{it}} \ge N_{\min} l_t && \forall t \in \mathcal{T}_L \label{eq:iii}\\
      & \textstyle\sum_{j=1}^d{(a^+_{jt}+a^-_{jt})} = 1 && \forall t \in \mathcal{T}_B \\
      & a_{jt} = a^+_{jt} - a^-_{jt} && \forall j \in [d], \quad \forall t \in \mathcal{T}_B \\
      & a^+_{jt} \le o_{jt} && \forall j \in [d], \quad \forall t \in \mathcal{T}_B \\
      & a^-_{jt} \le (1 - o_{jt}) && \forall j \in [d], \quad \forall t \in \mathcal{T}_B \\
            & z_{it}, l_{t} \in \{0, 1\} && \forall i \in [\nSample], \quad \forall t \in \mathcal{T}_L \\
            & o_{jt} \in \{0, 1\} && \forall j \in [\ndim], \quad \forall t \in \mathcal{T}_B \\
            & a_{jt}, b_{t} \in [-1, 1] && \forall j \in [\ndim], \quad \forall t \in \mathcal{T}_B \\
            & a_{jt}^{+}, a_{jt}^{-} \in [0, 1] && \forall j \in [\ndim], \quad \forall t \in \mathcal{T}_B \\
            & \phi_{it} \in \mathbb{R} && \forall i \in [\nSample], \quad \forall t \in \mathcal{T}_L
  \end{align}
\end{subequations}
The constraints \cref{eq:eee,eq:fff,eq:ggg,eq:hhh,eq:iii} are the optimal classification tree with hyperplanes (OCT-H) of \citet{bertsimasOptimalClassificationTrees2017}; the other constraints are our extension thereof.


\section{Axis-aligned regression formulation}
\label{appx:axisalign}
In this section, we introduce a variant of the affine hyperplane regression tree, presented in \cref{sec:mip-form}, that accommodates using hyperplanes aligned with cartesian for partitioning the space.
\Cref{table:hyperparam,table:variablesaffhyp} summarize the hyperparameters and variables of the mixed-integer formulation.

\textbf{Axis-aligned partition of the space.} At each branching node $m\in\bNode$, a hyperplane splits the space in two subspaces
\begin{equation}\label{eq:splitstrictappx}
  a_m^{\top} x_i - b_{m}  < 0 \qquad a_m^{\top} x_i - b_{m}  \ge 0,
\end{equation}
that will be associated to the left and right children of node $m$.
The parameters $a_{m}\in \{0,1\}^{\ndim}$ and $b_{m} \in [0,1]$ are variables of the mixed integer program.
In the axis-aligned formulation, the hyperplane is aligned with an axis of the cartesian space.
This is enforced by requiring for all branching node $m\in\bNode$ that $a_{m}$ take boolean values, only one of which is one:
\begin{equation}
  \textstyle\sum_{j=1}^d{a_{jm}} = 1, \qquad 0 \le b_m \le 1.
\end{equation}
Since $x_{i}\in[0, 1]$, there holds $a_{m}^{\top}x_{i}\in[0, 1]$, so that $b_{m}$ is constrained to $[0, 1]$ without loss of generality.

In order to model the strict inequality in \eqref{eq:splitstrictappx}, we follow \citet{bertsimasOptimalClassificationTrees2017} and introduce the vector $\epsilon\in\bbR^{\ndim}$ of smallest increments between two distinct consecutive values in points $(x_{i})_{i=[\nSample]}$ in any dimension:
\begin{equation}
  \epsilon_j = \min \left\{x_j^{(i+1)} - x_j^{(i)}, \text{ for } i \in [\ndim-1] \;:\; x_j^{(i+1)} \ne x_j^{(i)}\right\}
\end{equation}
where $x_j^{(i)}$ is the $i$-th largest value in the $j$-th dimension.
$\epsilon_{\max}$ is the highest value of $\epsilon_j$ and serves as a tight big-M bound, leading to the formulation
\begin{align}
a_m^{\top} x_i &\ge b_m - (1 - z_{it}) & \forall m \in A_R(t) \\ 
{a}_m^{\top} ({x}_i + {\epsilon}) &\le b_m + (1 + \epsilon_{\max})(1 - z_{it}) & \forall m \in A_L(t) 
\end{align}
both for all $i$ in $[\nSample]$ and all $t$ in $\mathcal{T}_L$. Recall that $z_{it}$ takes binary values and is equal to one if sample $x_i$ belongs to leaf node $t$. 

Combining these elements yields the axis-aligned formulation:
\begin{subequations}
  \label{eq:OrtFormulation}
  \begin{align}
    \min \; &
              \frac{1}{\nSample} \sum^{\nSample}_{i=1} \delta_i & \\
    \textrm{s.t.} \quad & \delta_i \ge \phi_{it} - (1 - z_{it})M && \forall i \in [\nSample], \quad \forall t \in \mathcal{T}_L \\
            & \delta_i \ge - \phi_{it} - (1 - z_{it})M && \forall i \in [\nSample], \quad \forall t \in \mathcal{T}_L \\
            & \phi_{it} = y_i - \textrm{poly}(x_i; c_t) && \forall i \in [\nSample], \quad \forall t \in \mathcal{T}_L \\
            & a_m^{\top} {x}_i \ge b_m - (1 - z_{it}) && \forall i \in [\nSample], \quad \forall t \in \mathcal{T}_L, \quad \forall m \in A_R(t)    \label{eq:axaligned_eee}\\
            & {a}_m^{\top} ({x}_i + {\epsilon}) \le b_m & & \notag\\
            & \quad + (1 + \epsilon_{\max})(1 - z_{it})& & \forall i \in [\nSample], \quad \forall t \in \mathcal{T}_L, \quad \forall m \in A_L(t)   \label{eq:axaligned_fff}\\
            & \textstyle\sum_{t \in \mathcal{T}_L}{z_{it}} = 1 && \forall i \in [\nSample]  \label{eq:aaa}\\
            & z_{it} \le l_t && \forall i \in [\nSample], \quad \forall t \in \mathcal{T}_L \label{eq:bbb}\\
            & \textstyle\sum_{i=1}^{\nSample}{z_{it}} \ge N_{\min} l_t && \forall t \in \mathcal{T}_L \label{eq:minN}\\
            & \textstyle\sum_{j=1}^d{a_{jt}} = 1 && \forall t \in \mathcal{T}_B \label{eq:ccc}\\
            & 0 \le b_t \le 1 && \forall t \in \mathcal{T}_B \label{eq:ddd}\\
            & z_{it}, l_{t} \in \{0, 1\} && \forall i \in [\nSample], \quad \forall t \in \mathcal{T}_L \label{eq:axaligned_lll}\\
            & a_{jt} \in \{0, 1\} && \forall j \in [\ndim], \quad \forall t \in \mathcal{T}_B \label{eq:axaligned_mmm}\\
            & \phi_{it} \in \mathbb{R} && \forall i \in [\nSample], \quad \forall t \in \mathcal{T}_L
  \end{align}
\end{subequations}
The constraints \cref{eq:axaligned_eee,eq:axaligned_fff,eq:aaa,eq:bbb,eq:minN,eq:ccc,eq:ddd,eq:axaligned_lll,eq:axaligned_mmm} are the optimal classification trees (OCT) of \citet{bertsimasOptimalClassificationTrees2017}, whereas the other constraints are our extensions thereof. Note that we do not use the complexity parameters of the OCT formulation ($d_t$) and replace them with 1, where appropriate.

\begin{table}[t]
  \caption{Summary of the variables of the axis-aligned regression tree formulation\label{table:variablesaxisaligned}}
  \centering
  \begin{tabular}{lll}
    \toprule
    variable & index domain & interpretation \\ \midrule
    $l_{t} \in \{0, 1\}$ & $t\in\lNode$ & 1 iff any point is assigned to leaf $t$ \\
    $z_{it} \in \{0, 1\}$ & $t\in\lNode$, $i\in[\nSample]$ & 1 iff point $x_{i}$ is assigned to leaf $t$ \\
    $a_{m} \in \{0, 1\}^{\ndim}$ & \multirow{2}{*}{$m\in\bNode$} & \multirow{2}{*}{coefficients of the axis-aligned cut} \\
    $b_{m} \in \bbR$ & & \\
    $\phi_{it} \in \bbR$ & $t\in\lNode$, $i\in[\nSample]$ & fit error of point $x_{i}$ by the polynomial of leaf $t$ \\
    $\delta_{i} \in \bbR$ & $i\in[\nSample]$ & fit error of point $x_{i}$ by the piecewise polynomial function \\
    $c_{t} \in \bbR^{\binom{\polyDeg + \ndim}{\ndim}} $ & $t\in\lNode$ & coefficients of the degree $\polyDeg$ polynomial associated with leaf $t$ \\
    \bottomrule
  \end{tabular}
\end{table}

\section{Additional details on \cref{sec:numer-exper}}
\label{sec:addit-deta-num}
In this section, we provide additional figures and details on the numerical experiments presented in \Cref{sec:numer-exper}.

\Cref{fig:NN_architecture} displays the architecture of the Neural Network used in \cref{sec:regr-tame-funct}.

\begin{figure}[t]
  \centering
  \includegraphics[width=0.6\columnwidth]{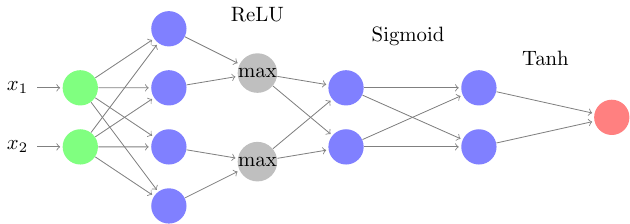}
  \caption{%
    The architecture of the example Neural Network. Input neurons are green, blue nodes are neurons in hidden layers, and the output neuron is red.
    Grey nodes represent the max-pooling layer.
    Above the connections are the names of activation functions used on the outputs of the layer to the left of each respective name.
  }
  \label{fig:NN_architecture}
\end{figure}

\Cref{fig:mip_process} shows the evolution of the optimization routine, as described in \cref{sec:note-optim-process}.

\begin{figure}[t]
    \centering
      \includegraphics[width=0.6\columnwidth]{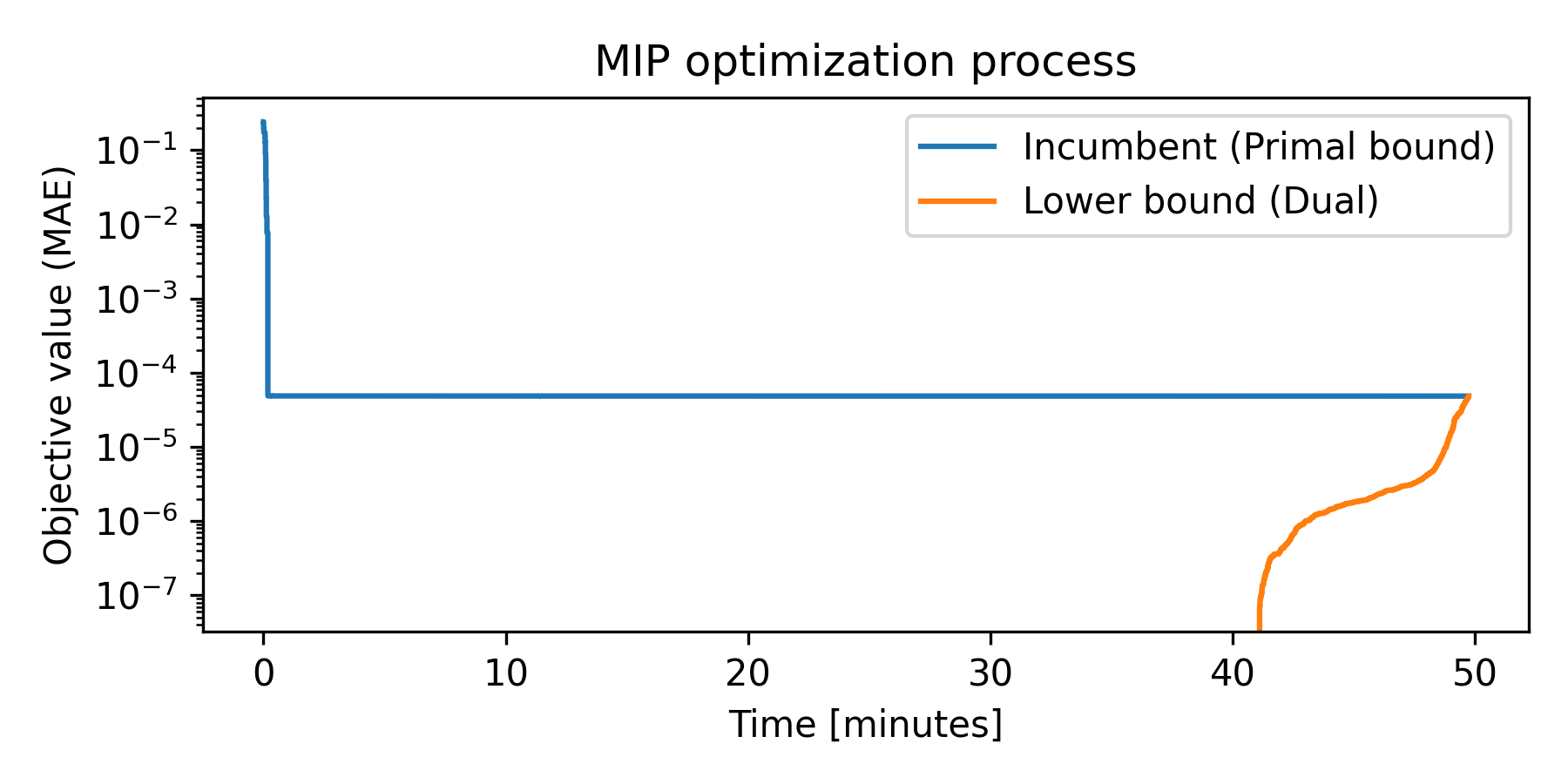}
    \caption{Objective value (mean absolute error) of the incumbent solution during the optimization. The full optimization takes 50 minutes, although a solution with almost optimal error is found within seconds.}
    \label{fig:mip_process}
\end{figure}

\begin{table}[t]
  \caption{%
    \label{table:expsregression}
    Normalized error between tame functions and affine-hyperplane tree approximations \eqref{eq:OrtFormulationHplane}.
  }
  \begin{center}
    \resizebox{0.48\textwidth}{!}{
      \begin{tabular}{lcccc}
        \toprule
        Function & depth & max. err. & mean err. & median err. \\ \midrule
        \multirow{2}{*}{$\fCone$ \eqref{eq:cone2dintro}} & 2 & \scinum{0.49887991165379475} & \scinum{0.04685379257004139} & \scinum{0.00010017119702482091} \\
                 & 3 & \scinum{0.3742223619441276} & \scinum{0.010007973010744676} & \scinum{6.371677832554213e-05} \\ 
        \midrule
        \multirow{2}{*}{NN} & 2
                         & \scinum{0.32716195149752386} & \scinum{0.036286603222585916} & \scinum{0.024595572191527453} \\
                 & 3
                         & \scinum{0.1886663085338612} & \scinum{0.012912554084793467} & \scinum{0.009264882675604167} \\
        \bottomrule
      \end{tabular}
    }
    \vskip -0.15in
  \end{center}
\end{table}

\begin{table}[p]
  \centering
  \caption{
    Normalized absolute error for the denoising scenarios in \Cref{fig:data_noisy}. The error is computed as the absolute difference between the ground truth signal and its approximation by the \emph{axis-aligned} trees, divided by the maximal value of the ground truth.
    \label{table:expsdenoising}
  }
  \resizebox{0.48\textwidth}{!}{
  \begin{tabular}{lcccccc}
    \toprule
       Denoising & max. err. & mean err. & median err. \\ \midrule
    Scenario 1 & \scinum{1.0975286704379505} & \scinum{0.14299021402955037} & \scinum{0.09257820874197198} \\
    Scenario 2 & \scinum{1.4545038074634236} & \scinum{0.18311634112276917} & \scinum{0.09003334246760586} \\
    Scenario 3 & \scinum{0.9546384046329512} & \scinum{0.11935686801291823} & \scinum{0.054727166940183425} \\
    Scenario 4 & \scinum{0.8339429435772073} & \scinum{0.04129052618009546} & \scinum{0.01588706426515545} \\
    \bottomrule
  \end{tabular}
  }
\end{table}


\section{Complementary experimental results}
\label{sec:numericalexperiments}
In this section, we give complementary experiments that illustrate the practical behavior of the axis-aligned and affine-hyperplane regression models.

\paragraph{Setup.}
The setup is identical to the one described in \cref{sec:numer-exper}.
We consider three additional regression problems, for which we plot the landscapes in the forthcoming figures and give approximation errors in \cref{table:expsappx}.




\paragraph{Piecewise-linear norms.}
\label{sec:piec-line-norms}

We consider two simple piecewise linear test functions: the $\|\cdot\|_{1}$ and $\|\cdot\|_{\infty}$ norms, defined for $x\in\bbR^{\ndim}$ by
\begin{equation}
  \|x\|_{1} = \sum_{i=1}^{\ndim} |x_{i}|, \qquad \|x\|_{\infty} = \max_{i\in [\ndim]} |x_{i}|.
\end{equation}
For both functions, we set depth $D=2$ and polynomial degree $\polyDeg=1$. We sample $\nSample=250$ points in the approximation space. We test both the axis-aligned \eqref{eq:OrtFormulation} and the general affine-hyperplane \eqref{eq:OrtFormulationHplane} formulations with a time limit of 5 minutes for each optimization.
For both norms, the axis-aligned formulation was solved to optimality and the affine-hyperplane formulation timed out.

\Cref{fig:l1norm} shows the results on the $\|\cdot\|_{1}$ norm. 
Note that the full-dimensional cells of the $\|\cdot\|_{1}$ are axis aligned, so the axis-aligned formulation \eqref{eq:OrtFormulation} (with $D=2$) recovers both the correct cells and the correct polynomial function on each piece.
The more general affine-hyperplane formulation \eqref{eq:OrtFormulationHplane} with performs equally well. 


\Cref{fig:linfnorm} shows the results for the $\|\cdot\|_{\infty}$.
The axis-aligned formulation \eqref{eq:OrtFormulation} with depth yields a piecewise polynomial function that performs poorly at approximating the function.
This is reasonable, as the full-dimensional cells are not axis-aligned anymore.
The affine-hyperplane formulation \eqref{eq:OrtFormulationHplane} yields a piecewise polynomial function that matches the cells of the function, as well as the polynomial expression of the function on the cells.

Numerical results for both norms are in \Cref{table:expsappx}. They show a slight increase in error when using the affine-hyperplane formulation on the $\|\cdot\|_1$ norm. 
This can be attributed to the fact that the true partitioning is axis-aligned, which agrees with the main limitation of the axis-aligned formulation. 
And because the axis-aligned formulation is simpler to optimize, we obtain a provably optimal solution. Despite that, the solution of affine-hyperplane formulation has only slightly worse error. Additionally, looking at the errors for the $\|\cdot\|_{\infty}$ norm, we see order(s) of magnitude improvements in the error, when using the affine-hyperplane formulation. This underlines the increased expressivness of the more general formulation.

\begin{figure*}[p]
    \centering
    \begin{subfigure}{.6\textwidth}
        \includegraphics[width=\textwidth]{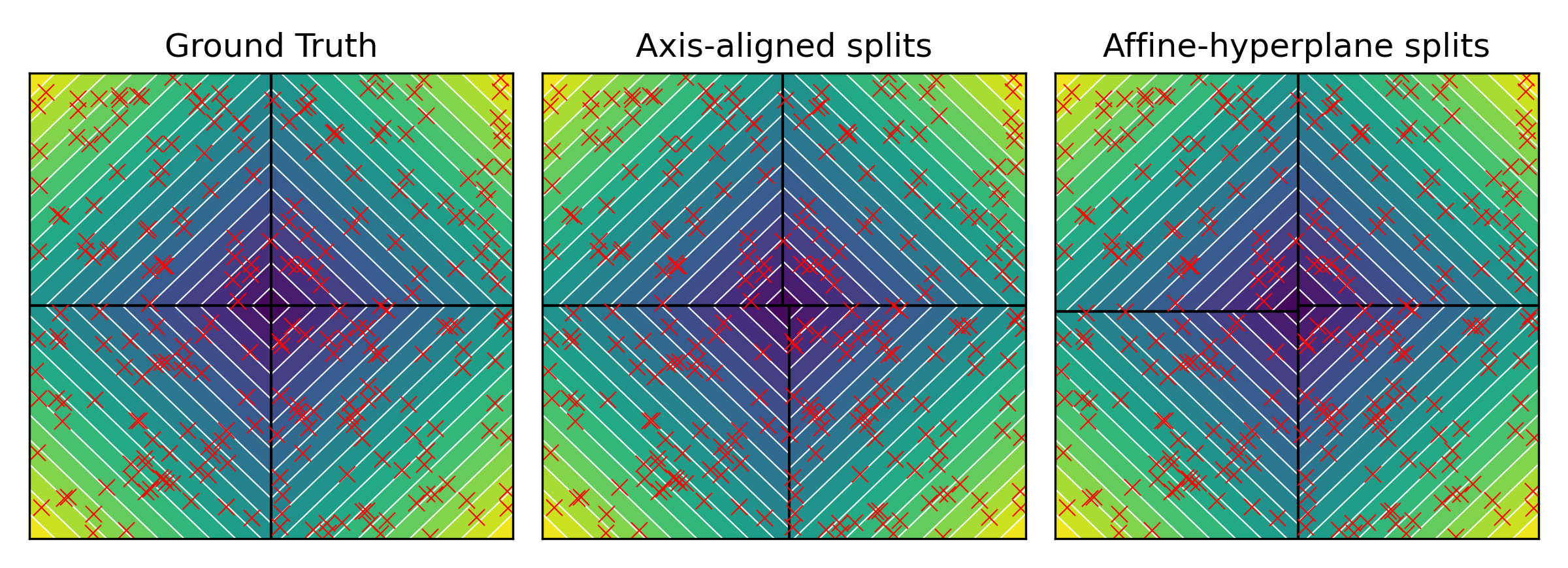}
        \caption{$\|\cdot\|_1$ norm.}
        \label{fig:l1norm}
    \end{subfigure}
    \begin{subfigure}{.6\textwidth}
    
        \includegraphics[width=\textwidth]{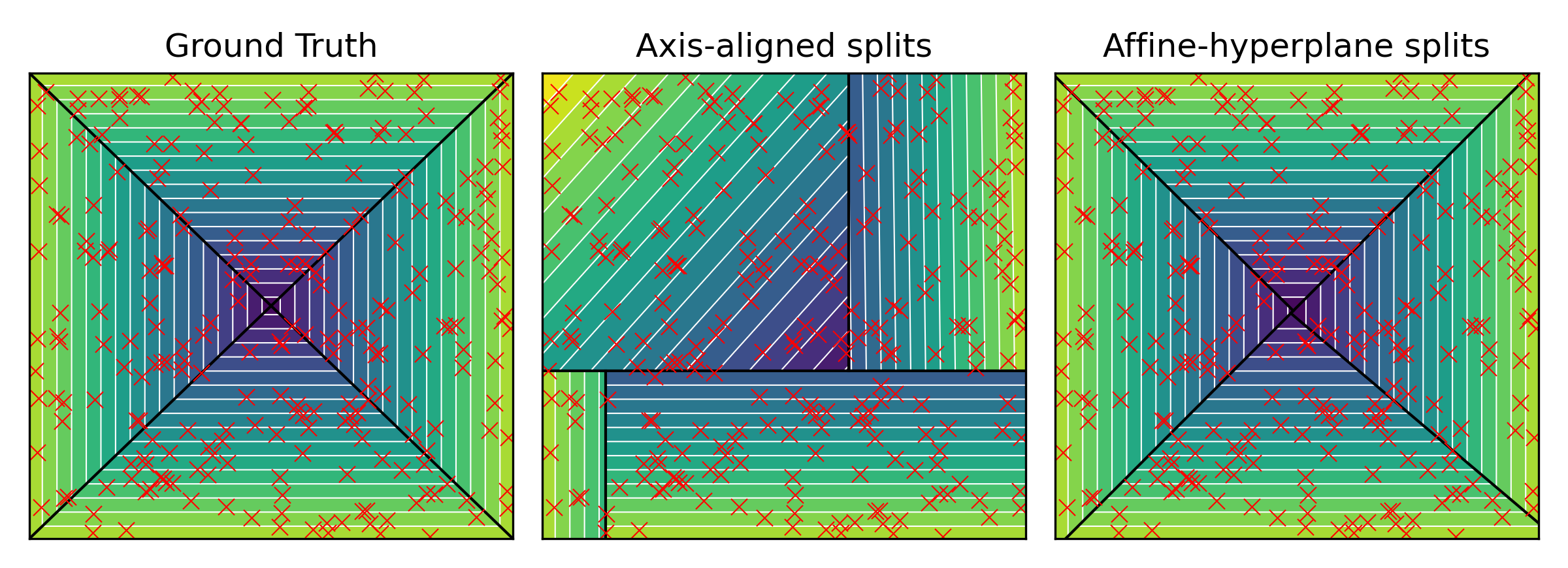}
        \caption{$\|\cdot\|_\infty$ norm.}
        \label{fig:linfnorm}
    \end{subfigure}
    \caption{Results on norm functions. On the left are the original functions, middle column contains the results of the axis-aligned formulation \eqref{eq:OrtFormulation} and on the right are results of the more general affine-hyperplane formulation \eqref{eq:OrtFormulationHplane} from the main body of the paper. Red crosses are the coordinates of samples used in the optimization. White lines are the level lines of the function values. Black lines show the partitioning of the space.}
\end{figure*}

\paragraph{An additional Neural Network approximation}

Much as described in \cref{sec:numer-exper}, we consider a similar NN with $25$ parameters (2 hidden layers with 4 and 2 neurons respectively), but with \emph{only ReLU} activation functions.
The network is trained to minimize the mean squared error with the 2-dimensional function $x\mapsto  2\sin{x_1} + 2\cos{x_2} + x_2 / 2$ taken at $10$ random points from the input space.
All data is processed in a single batch, for 5000 epochs, using \texttt{AdamW} optimizer.

The piecewise linear approximation is obtained from the affine-hyperplane formulation with depth $D=2$ and $\nSample=225=15^{2}$ sample points (shown as red crosses) placed on a regular grid of $15\times 15$ points. The degree of the polynomial pieces is $\polyDeg = 1$. The MIP optimization times out after 5 minutes.

\Cref{fig:nn_relu} shows the landscape of the network in the left pane and the piecewise linear approximation in the right pane.
The obtained piecewise polynomial function essentially recovers the cell decomposition of the network.
\Cref{fig:nn_errors} presents the difference between the NN oputut and the approximation.
The approximation recovers the slope of the network correctly on each of the cells.
The discrepancy between the two functions is caused by the slight mismatch between the cell decomposition of the network and its approximation, which could arguably be improved by taking samples from a denser grid.

Looking at the last row of \Cref{table:expsappx}, we notice that the median error is the lowest among all functions by orders of magnitude, pointing to the high quality of the fit of the polynomials in each partition. This might be due to the properties of taking points on a regular grid which might allow for better approximation.

\begin{table*}[p]
  \caption{Normalized absolute error between the functions and
their approximations by the axis-aligned \eqref{eq:OrtFormulation} and affine-hyperplane trees \eqref{eq:OrtFormulationHplane} of depth 2. The error is computed on a $1000 \times 1000$ grid of regularly
spaced points. We divide the absolute errors by the maximal absolute value of the underlying ground truth to improve comparability.\label{table:expsappx}}
  \centering
  \begin{tabular}{lcccccc}
    \toprule
    \multirow{2}{*}{Test function} & \multicolumn{3}{c}{Axis-aligned tree \eqref{eq:OrtFormulation}} & \multicolumn{3}{c}{Affine-hyperplane tree \eqref{eq:OrtFormulationHplane}} \\
                                   & max. err. & mean err. & median err. & max. err. & mean err. & median err. \\ \midrule
    $\|\cdot\|_{1}$ & \scinum{0.022015681793301364} & \scinum{9.788552931962676e-05} & \scinum{3.0399942109511357e-05} & \scinum{0.02604307023189877} & \scinum{0.00011084159563612121} & \scinum{3.0358761727917738e-05} \\
    $\|\cdot\|_{\infty}$ & \scinum{0.7186999470693821} & \scinum{0.08506899722574107} & \scinum{0.028201825194180685} & \scinum{0.06604959284704932} & \scinum{0.0004217446020628828} & \scinum{5.04670190798251e-05} \\
    ReLU NN & - & - & - & \scinum{0.030146090044253088} & \scinum{0.00017965119982446732} & \scinum{2.686642996129214e-07} \\
    \bottomrule
  \end{tabular}
\end{table*}

\begin{figure}[p]
    \centering
    \begin{subfigure}{.45\textwidth}
        \centering
        \includegraphics[height=0.5\textwidth]{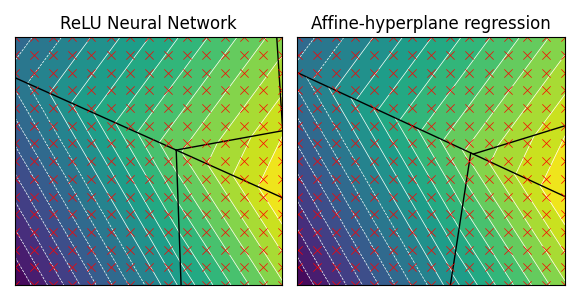}
        \caption{Approximation result. On the left is the NN output landscape and on the right is a resulting tree approximation. The black lines show the partitioning of the space, white lines are level lines, and red crosses show coordinates of sampled points.}
        \label{fig:nn_relu}
    \end{subfigure}
    \hspace{.03\textwidth}
    \begin{subfigure}{.45\textwidth}
        \centering
        \includegraphics[height=0.5\textwidth]{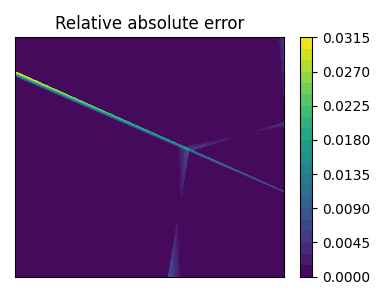}
        \caption{Normalized absolute error of the ReLU NN approximation by the tree model. The errors are divided by the maximum absolute value of the NN output. We can see that the majority of the error is due to slight inaccuracies in the partitioning. These errors are described numerically in \Cref{table:expsappx}.}
        \label{fig:nn_errors}
    \end{subfigure}
    \caption{Approximation of the Neural Network with only ReLU activations. On the left, in \Cref{fig:nn_relu}, is the original output and the approximation. \Cref{fig:nn_errors}, on the right, shows the normalized absolute error between the two functions.}
\end{figure}
\section*{Acknowledgments}
\begingroup
\small
This work has received funding from the European Union’s Horizon Europe research and innovation programme under grant agreement No. 101070568.
This work received funding from the National Centre for Energy II (TN02000025).
\endgroup

\bibliographystyle{plainnat}
\bibliography{refs.bib}



\end{document}